\documentclass[oneside,reqno]{amsart}
\usepackage{amsmath,amssymb}
\usepackage{amsthm}

\usepackage[a4paper]{geometry}
\usepackage{hyperref}
\hypersetup{
    colorlinks=true,
    linkcolor=black,          
    citecolor=black,        
    filecolor=magenta,      
    urlcolor=cyan,
}
\usepackage{xcolor}

\allowdisplaybreaks

\numberwithin{equation}{section}
\newtheorem{theorem}{Theorem}[section]
\newtheorem{lemma}[theorem]{Lemma}
\newtheorem{corollary}[theorem]{Corollary}
\newtheorem{proposition}[theorem]{Proposition}

\theoremstyle{definition}
\newtheorem{definition}[theorem]{Definition}
\newtheorem*{definition*}{Definition}

\newtheorem{example}[theorem]{Example}
\newtheorem{conjecture}[theorem]{Conjecture}

\theoremstyle{remark}
\newtheorem{remark}[theorem]{Remark}
\newtheorem*{remark*}{Remark}

\title{Criteria for apwenian sequences}
\author{Yingjun Guo}
\address[Y.-J. Guo]{College of Science, Huazhong Agricultural University, 430070 Wuhan, China}
\email{guoyingjun2005@126.com}
\author{Guoniu Han}
\address[G.-N. Han]{Institute de Recherche Math\'{e}matique Avanc\'{e}e, Universit\'{e} de Strasbourg et CNRS, 7 rue Ren\'{e} Descartes, 67084 Strasbourg, France}
\email{guoniu.han@unistra.fr}
\author{Wen Wu}
\thanks{Wen Wu is the corresponding author.}
\address[W. Wu]{School of Mathematics, South China University of Technology, 510640 Guangzhou,  China}
\email[Corresponding author]{wuwen@scut.edu.cn}
\date{}

\begin{document}
\begin{abstract}
In 1998, Allouche, Peyri\`{e}re, Wen and Wen showed that the Hankel determinant $H_n$ of the Thue-Morse sequence over $\{-1,1\}$ satisfies $H_n/2^{n-1}\equiv 1~(\mathrm{mod}~2)$ for all $n\geq 1$. Inspired by this result, Fu and Han introduced \emph{apwenian} sequences over $\{-1,1\}$, namely, $\pm 1$ sequences whose Hankel determinants satisfy $H_n/2^{n-1}\equiv 1~(\mathrm{mod}~2)$ for all $n\geq 1$, and proved with computer assistance that a few sequences are apwenian. In this paper, we obtain an easy to check criterion for apwenian sequences, which allows us to determine all apwenian sequences that are fixed points of substitutions of constant length. Let  $f(z)$ be the generating functions of such apwenian sequences. We show that for all integer $b\ge 2$ with $f(1/b)\neq 0$, the real number $f(1/b)$ is transcendental and its irrationality exponent is equal to $2$.
	
Besides, we also derive a criterion for $0$-$1$ apwenian sequences whose Hankel determinants satisfy $H_n\equiv 1~(\mathrm{mod}~2)$ for all $n\geq 1$. We find that the only $0$-$1$ apwenian sequence, among all fixed points of substitutions of constant length, is the period-doubling sequence. Various examples of apwenian sequences given by substitutions with projection are also provided. Furthermore, we prove that all Sturmian sequences over $\{-1,1\}$ or $\{0,1\}$ are not apwenian. And  we conjecture that fixed points of substitution of non-constant length over $\{-1,1\}$ or $\{0,1\}$ can not be apwenian.
\end{abstract}
\maketitle

\setcounter{tocdepth}{1}
\tableofcontents

\goodbreak

\section{Introduction} 
Let $\mathbf{a}=a_0a_1a_2\dots$ be a sequence over a field $\mathbb{F}$ and let $f(z)=\sum_{n=0}^{\infty}a_nz^{n}$ be its generating function. The Hankel determinant of the sequence $\mathbf{a}$ (or the formal power series $f$) of order $n$ $(n\geq 1)$ is
\begin{align*}
H_n(\mathbf{a})
& = \left|\begin{matrix}
a_{0} & a_{1} & \cdots & a_{n-1}\\
a_{1} & a_{2} & \cdots & a_{n}\\
\vdots & \vdots & \ddots & \vdots\\
a_{n-1} &  a_{n} & \cdots & a_{2n-2}
\end{matrix}\right| = \det(a_{i+j})_{0\leq i,j\leq n-1}.
\end{align*}
Finding rational approximations for $f(z)$ is strongly related to the Hankel determinants of $f$. 
Let $p$ and $q$ be nonnegative integers. The Pad\'e approximant $[p/q]_f(z)$ of $f(z)$ is a rational function $P(z)/Q(z)$ with $P(z),\,Q(z)\in\mathbb{F}[z]$, $Q(0)=1$ and $\deg P(z)\leq p,\, \deg Q(z)\leq q$ such that \[f(z)-\frac{P(z)}{Q(z)}=O(z^{p+q+1}).\] 
Actually, the non-vanishing of $H_n(\mathbf{a})$ guarantees the existence of the Pad\'e approximant $[(n-1)/n]_f(z)$ (see \cite[p. 34-36]{Bre80}) and \[f(z)-[(n-1)/n]_f(z)=\frac{H_{n+1}(\mathbf{a})}{H_n(\mathbf{a})}z^{2n} + O(z^{2n+1}).\] 
Let $b$ be a non-zero integer such that $f(1/b)$ converges. Then $[(n-1)/n]_{f}(1/b)$ yields a rational approximation for the real number $f(1/b)$, i.e.
\[\bigl|f(1/b)-[(n-1)/n]_{f}(1/b)\bigr|\leq \frac{c(n)}{b^{2n}}.\]
However, sometimes $c(n)$ may become even bigger than $b^{2n}$. Therefore Pad\'e approximants of $f(z)$ do not automatically give us good approximations of $f(1/b)$. But in some case (like in the case of infinite products, see for example \cite{D19,Bu11,BHYW16}) they do. 

When $f(z)$ satisfies certain algebraic equations, in a sequence of works, the Hankel determinants of $f(z)$ have been shown to be an efficient tool in estimating the irrationality exponent of $f(1/b)$. Based on several Pad\'{e} approximants, Adamczewski and Rivoal \cite{AR09} estimated the irrationality exponents for some automatic real numbers  using Mahler's method. In 2011, Bugeaud \cite{Bu11} proved that the irrationality exponent of the Thue-Morse number is equal to $2$. In his proof, one important requirement is that all the Hankel determinants of the Thue-Morse sequence on $\{-1,1\}$ are non-vanishing. This result was obtained in 1998 by Allouche, Peyri\`{e}re, Wen and Wen \cite{APWW98}. Bugeaud's method works well for degree $2$ Mahler functions. The remaining difficulty is to calculate the Hankel determinants. Coons \cite{Co13} proved that the irrationality exponent of the sum of the reciprocals of the Fermat numbers is 2. Guo, Wen and Wu \cite{GWW14} verified that the irrationality exponents of the regular paper-folding numbers are $2$. Wen and Wu \cite{WW14} showed that the irrationality exponents of the Cantor numbers are also $2$. The idea of evaluating the Hankel determinants in these works are the same as in \cite{APWW98}. Han \cite{Han16} proved that
the Hankel determinants of a large family of sequences are non-zero by using
the Hankel continued fractions. Using Han's result, Bugeaud, Han, Wen and Yao \cite{BHYW16} estimated irrationality exponents of $f(1/b)$ for a large class of Mahler functions $f(z)$, provided that the distribution of indices at which Hankel determinants of $f(z)$ do not vanish is known. For a class of Mahler functions $f(z)$,  Badziahin \cite{DB19} developed a useful theorem which can be used to compute the exact value of the irrationality exponent for $f(b)$ as soon as the continued fraction for the corresponding Mahler function is known. This improved the result of Bugeaud, Han, Wen and Yao \cite{BHYW16}. 

\subsection{$\pm 1$ apwenian sequences}
In the seminal work \cite{APWW98}, Allouche, Peyri\`{e}re, Wen and Wen showed that the Hankel determinants of the Thue-Morse sequence $\mathbf{t}$ over $\{-1,1\}$ satisfy  for all $n\geq 1$,
\[H_{n}(\mathbf{t})/2^{n-1}\equiv 1\quad (\mathrm{mod}~2).\]
\begin{definition*}[$\pm 1$ Apwenian sequence]
A sequence $\mathbf{d}=d_0d_1d_2\cdots \in\{-1,1\}^{\infty}$ is called an \emph{apwenian} sequence if for all $n\geq 1$, \[\frac{H_{n}(\mathbf{d})}{2^{n-1}}\equiv 1 \quad (\mathrm{mod}~2).\]
When $\mathbf{d}\in\{-1,1\}^{\infty}$ is apwenian, then we also say that its generating function $f(z)$ is \emph{apwenian}.
\end{definition*}

\begin{remark*}
The apwenian sequences occur in pairs. Namely, if $\mathbf{d}$ is apwenian, then apparently, $-\mathbf{d}$ is also apwenian. In the sequel, we only focus on the apwenian sequences starting with $1$.
\end{remark*}

Fu and Han \cite{FH16} introduced the above \emph{apwenian} sequences in honour of the authors of \cite{APWW98},
and investigated Hankel determinants for the formal power series
\begin{equation}\label{eq:fh}
f_p(z)=\prod_{i=0}^{\infty}P(z^{p^{i}})
\end{equation}
where $p\geq 2$ is an integer, $P(z)=v_0+v_{1}z+\cdots v_{p-1}z^{p-1}$, $v_0=1$ and $v_{i}\in\{-1,1\}$ for all $i=1,\ldots,p-1$. They checked for prime numbers $p\leq 17$ with computer assistance and found that the apwenian series satisfying \eqref{eq:fh} are quite rare:
\[
\begin{array}{c|*{7}{c}}
p & 2 & 3 & 5 & 7 & 11 & 13 & 17\\
\hline
N_p^{\prime} & 1 & 1 & 1 & 0 & 1 & 1 & 2
\end{array}
\]
where $N_p^{\prime}$ is the number of apwenian series satisfying \eqref{eq:fh} and $v_1=-1$. They conjectured in \cite{FH16} that $N_{19}^{\prime}=1$.
In this paper, we obtain an easy to check criterion for apwenian sequences,
	which allows us to
	determine all apwenian sequences that are fixed points of substitutions of constant length.

\begin{theorem}[$\pm1$ Criterion]\label{thm:A}
Let $\mathbf{d}=d_0d_1d_2\cdots \in\{-1,1\}^{\infty}$. Then $\mathbf{d}$ is apwenian if and only if
\begin{equation}\label{eq:-01}
\forall\, n\geq 0,\quad \frac{d_n+d_{n+1}-d_{2n+1}-d_{2n+2}}{2} \equiv 1\pmod{2}.
\end{equation}
\end{theorem}

Let $\mathbf{d}\in\{-1,1\}^{\infty}$ and $f(z)=\sum_{n=0}^{\infty}d_{n}z^{n}$. In fact, if $f(z)$ satisfies \eqref{eq:fh}, then $\mathbf{d}$ is the fixed point of the following substitution of length $p$:
\begin{equation}\label{intro:eq-1}
1\mapsto v_0v_1\cdots v_{p-1},\quad -1\mapsto \bar{v}_0\bar{v}_1\cdots \bar{v}_{p-1}
\end{equation}
where $v_0=1$, $v_i\in\{-1,1\}$ for $i=1,\dots,p-1$ and $\bar{x}:=-x$ for all $x\in\{-1,1\}$.
In general, we study the fixed point
$\mathbf{d}=\lim_{n\to+\infty}\sigma^{n}(1)$
of substitution $\sigma$ of constant length $p$ (not necessarily prime numbers):
\begin{equation}\label{eq:2nd-1}
	\sigma: 1\mapsto v_0v_1\cdots v_{p-1},\quad -1\mapsto w_0w_1\cdots w_{p-1}
\end{equation}
where $v_i,\, w_i\in\{-1,1\}$ for all $i=0,1,\dots,p-1$.
By using the $\pm1$ criterion, we are able to show that

\begin{theorem}\label{thm:B}
Let $\mathbf{d}=\lim\limits_{n\to\infty}\sigma^{n}(1)$ where $\sigma$ is given in \eqref{eq:2nd-1} with $v_0=1$ and $f(z)=\sum_{n=0}^{\infty}d_nz^n$. If $\mathbf{d}$ is apwenian, then $f(z)$ satisfies \eqref{eq:fh}.
\end{theorem}

In Theorem \ref{thm:t2} we establish a criterion on the substitution itself to tell when the series $f_{p}(z)$ is apwenian for all $p\geq 2$. This result can be applied to the calculation of irrationality exponents. Let $\xi$ be an irrational number. Its \emph{irrationality exponent} is the supremum of the real numbers $\nu$ such that 
\[\left|\xi-\frac{r}{s}\right|<\frac{1}{s^{\nu}}\]
holds for infinitely many pairs of $(r,s)\in\mathbb{Z}\times\mathbb{N}$. 
For the series $f_p(z)$ satisfying \eqref{eq:fh}, Badziahin~\cite{DB19} showed that the irrationality exponent of $f_{p}(1/b)$ is rational, where $b\ge 2$ is an integer with $P(1/b^{p^{m}})\neq 0$ for all integers $m\ge 0$. When the coefficient sequence of $f_{p}(z)$ is apwenian, the Hankel determinants of $f(z)$ are all non-vanishing. Then a combination of Badziahin's results \cite[Corollary 1]{D19} and \cite[Theorem 1.2]{DB19} yields that the irrationality exponent of $f_p(1/b)$ is equal to $2$.
We can also apply Bugeaud, Han, Wen and Yao's result \cite[Theorem 2.1]{BHYW16} to determine the irrationality exponent exactly. Further efforts and new methods are needed to examine the irrationality exponents of values of $f(z)$ at rational points. A consequence of \cite[Theorem 1.2]{DB19} (or \cite[Theorem 2.1]{BHYW16}) and Theorem \ref{thm:t2} is the following 
\begin{theorem}\label{thm:irrationality}
Let $p\geq 3$ be an odd number, $v_0=1$ and $v_i\in\{-1,1\}$ for $i=1,\dots,p-1$. Let $f_p(z)$ be the formal power series given in \eqref{eq:fh}. Assume that $b\ge 2$ is an integer such that $P(\frac{1}{b^{p^{i}}})\neq 0$ for all integers $i\ge 0$. For $m\ge p$, define $v_m:=v_i$ with $m\equiv i~(\mathrm{mod}~p)$ and $0\le i<p$. If \[\frac{v_j+v_{j+1}-v_{2j+1}-v_{2j+2}}{2}  \equiv 1 \pmod{2},\quad 0\leq j\leq p-2,\]
then the real number $f_p(1/b)$ is transcendental and its irrationality exponent is equal to $2$.
\end{theorem}

The next natural question is to determine the number of apwenian sequences (starting with $1$) which satisfy \eqref{intro:eq-1} for a given $p$. 
Let $N_p$ be the number of apwenian series $f_p(z)$ with $f_p(0)=1$. We remark that when $p$ is odd, $N_p=2N_p^{\prime}$. The first values of $N_p$ are the following:
\[
\begin{array}{c|*{19}{c}}
	p & 2 & 3 & 4&  5 & 6 & 7 & 8 & 9 &10 & 11 &12& 13 &14& 15&16& 17&18&19\\
\hline
	N_p & 1 & 2 & 1 &2& 0 &0 & 1 & 4 & 0 & 2 & 0 & 2 & 0 & 16 & 1 & 4 & 0 & 2
\end{array}.
\]
Theorem \ref{thm:t2} shows that when $p$ is even, the sequence is apwenian if and only if it is the Thue-Morse sequence. When $p$ is odd, the number is given by the next theorem which is a consequence of  Propositions \ref{prop:1} and \ref{prop:21}. For this, we need some notation. Denote by $\tau$ the permutation
\[\tau: \begin{pmatrix}
0& 1 &\cdots& \frac{p-3}{2}& \frac{p-1}{2}& \frac{p+1}{2}& \cdots& p-2\\
1& 3& \cdots& p-2& 0& 2& \cdots& p-3
\end{pmatrix}.\]
Let $p\geq 3$ be an odd number and set \[\mu(p):=\mathrm{ord}_p(2)=\min\{j\in [1,p-1] :  p|(2^{j}-1)\},\]
where $[s,t]$ denotes the set of integers $j$ such that $s\leq j\leq t$. 

\begin{theorem}\label{thm:D}
Let $p\geq 3$ be an odd number. If there exists an odd cycle in the cycle decomposition of $\tau$, then \[N_p=0;\] otherwise \[N_p=2^{k},\] where $k=\frac{1}{\mu(p)}\sum_{j=0}^{\mu(p)-1}\gcd(2^{j}-1,p) -1$ is the number of cycles of $\tau$.
\end{theorem}

In fact, our method does not only give the number of apwenian sequences, but it also provides a way to find those apwenian sequences (see Remark \ref{re:p=9}).

\subsection{$0$-$1$ apwenian sequences} While we consider the Hankel determinants modulo $2$ of a sequence $\mathbf{c}=c_0c_1c_2\cdots$, the natural alphabet that $\mathbf{c}$ lives on is $\{0,1\}$. One could project a sequence $\mathbf{d}\in\{-1,1\}^{\infty}$ to $\{0,1\}^{\infty}$ by considering the sequence
\[\tilde{\mathbf{d}}:=\left(\frac{d_i-d_{i+2}}{2}\ \,\bmod 2\right)_{i\geq 0}.\]
Our Lemma \ref{lem:01+-1} shows that $\mathbf{d}$ is $\pm 1$ apwenian if and only if $H_n(\tilde{\mathbf{d}})\equiv 1~(\mathrm{mod}~2)$ for all $n\geq 1$.  For this reason, we introduce the $0$-$1$ apwenian sequences. Throughout the paper, $0$-$1$ apwenian sequences and $\pm 1$ apwenian sequences are both called apwenian sequences without any ambiguous.
\begin{definition*}[$0$-$1$ Apwenian sequence]
A sequence $\mathbf{c}=c_0c_1c_2\cdots \in\{0,1\}^{\infty}$ is called an \emph{apwenian} sequence if for all $n\geq 1$, \[H_{n}(\mathbf{c})\equiv 1\pmod{2}.\]
\end{definition*}
A criterion for the $0$-$1$ apwenian sequences is also available.

\begin{theorem}[0-1 Criterion]\label{thm:E}
	Let $\mathbf{c}=c_0c_1c_2\cdots \in\{0,1\}^{\infty}$ such that $c_0=1$. Then $\mathbf{c}$ is apwenian if and only if for all $n\geq 0$,
	\begin{equation}\label{eq:01}
c_{n}\equiv c_{2n+1}+c_{2n+2}\pmod{2}.
	\end{equation}
\end{theorem}

We ask the same question on the number of $0$-$1$ apwenian sequences that are fixed points of substitutions of constant lengths. We have seen that there are not so many $\pm 1$ apwenian sequences. It is amazing that the $0$-$1$ apwenian sequences are even rarer. In detail, we study the fixed point
$\mathbf{c}=\lim_{n\to+\infty}\sigma_0^{n}(1)$
of substitution $\sigma_0$ of constant length $p$ (not necessarily prime):
\begin{equation}\label{eq:sigma0}
	\sigma_0: 1\mapsto v_0v_1\cdots v_{p-1},\quad 0\mapsto w_0w_1\cdots w_{p-1}
\end{equation}
where $v_0=1$, $v_i,\, w_i\in\{0,1\}$ for all $i=0,1,\dots,p-1$.
Denote by  \[\mathcal{O}:=\{\mathbf{c}\in\{0,1\}^{\infty} :  \mathbf{c}=\sigma_0(\mathbf{c}) \text{ for some }\sigma_0\text{ given by \eqref{eq:sigma0}}\}\]
the set of all the fixed points of substitutions of constant length.
\begin{theorem}\label{thm:F}
The only apwenian sequence in $\mathcal{O}$ is the period-doubling sequence given by the substitution $1\mapsto 10$ and $0\mapsto 11$.
\end{theorem}

\subsection{Organization}
In section \ref{sec:pre}, we list some notation and basic properties of substitutions and Jacobi continued fractions. In section \ref{sec:cri}, we prove Theorems \ref{thm:A} and \ref{thm:E} which give criteria for 0-1 and $\pm 1$ sequences to be apwenian. In section \ref{sec:01}, we prove Theorem \ref{thm:F}. In section \ref{sec:1-1}, we study a special class of $\pm 1$ sequences and prove Theorems \ref{thm:t2} and \ref{thm:D}. In section \ref{sec:gen}, we investigate all substitutions of constant length on $\{-1,1\}$ and prove Theorem \ref{thm:B}. In section \ref{sec:dis}, we give some examples. In the last section, we give some remarks on the permutation $\tau$ and apwenian sequences.

\section{Preliminary}\label{sec:pre} 
\subsection{Substitutions}
The basic notation of words and substitutions can be found in \cite{AS03,1794}.
Let $\mathcal{A}=\{a,b\}$ be an alphabet of two letters. For $n\geq 1$, the elements $w\in\mathcal{A}^{n}$ are called \emph{words}. The \emph{length} of a word $w$ is denoted by $|w|$. That is for any $w\in\mathcal{A}^{n}$, we have $|w|=n$. The set of all finite words on $\mathcal{A}$ is written $\mathcal{A}^{*}=\cup_{n\geq 0}\mathcal{A}^{n}$ where $\mathcal{A}^{0}=\{\varepsilon\}$ and $\varepsilon$ is the empty word.  Let $|w|_a$ be the number of occurrences of the letter $a$ in $w$.
The \emph{conjugate} of the letter $a$ (resp. $b$) is $b$ (resp. $a$), denoted by $\bar{a}$ (resp. $\bar{b}$). For any finite word $w\in\mathcal{A}^*$, we define by $\bar{w}=\bar{w}_1\cdots\bar{w}_{|w|}$ the conjugate of $w$.

For any two words $w,v\in \mathcal{A}^{*}$, their \emph{concatenation}, denoted by $wv$, is the word \[w_1w_2\dots w_{|w|}v_1v_2\dots v_{|v|}.\] For $u,v,w\in\mathcal{A}^{*}$, if $u=wv$, then $w$ (resp. $v$) is a \emph{prefix} (resp. \emph{suffix}) of $u$. For any $w,v\in\mathcal{A}^{*}$, the \emph{longest common prefix} of $w$ and $v$ is written $w\wedge v$. The set $\mathcal{A}^{*}$ with the operation concatenation is a free monoid. A \emph{substitution} $\sigma$ on the alphabet $\mathcal{A}$ is a mapping from $\mathcal{A}$ to $\mathcal{A}^{*}$ and it can be extended to a morphism on $\mathcal{A}^{*}$. Namely, for any $w,v\in\mathcal{A}^{*}$,
\[\sigma(wv)=\sigma(w)\sigma(v).\]
A substitution $\sigma$ is said of \emph{constant length}, if there exists $\ell\in\mathbb{N}$ such that $|\sigma(a)|=\ell$ for all $a\in\mathcal{A}$.

The elements $\mathbf{a}=a_0a_1a_2\cdots\in \mathcal{A}^{\infty}$ are infinite words or \emph{sequences}. For $\mathbf{a}\in\mathcal{A}^{\infty}$ and $w\in\mathcal{A}^{*}$, if $\mathbf{a}=vw\mathbf{a}^{\prime}$ where $v\in\mathcal{A}^{*}$ and $\mathbf{a}^{\prime}\in \mathcal{A}^{\infty}$, then $w$ is a \emph{subword} (or a \emph{factor}) of $\mathbf{a}$. A sequence $\mathbf{a}\in\mathcal{A}^{\infty}$ is \emph{uniformly recurrent} if for any subword $w$ of $\mathbf{a}$, there exists an integer $\ell$ such that every subword of length $\ell$ of $\mathbf{a}$ contains $w$ as a subword.

An infinite sequence $\mathbf{a}\in\mathcal{A}^{\infty}$ is called a \emph{fixed point} of a substitution $\sigma$ if $\sigma(\mathbf{a})=\mathbf{a}$.  There is a {\it natural metric} on $\mathcal{A}^{\infty}$: for $\mathbf{a},\,\mathbf{b}\in\mathcal{A}^{\infty}$,
\[\mathrm{dist}(\mathbf{a},\mathbf{b})=2^{-|\mathbf{a}\wedge\mathbf{b}|}.\]
A substitution $\sigma$ on $\mathcal{A}$ is \emph{prolongable} if there exists an $a\in\mathcal{A}$ such that $a$ is a prefix of $\sigma(a)$ and $|\sigma(a)|\geq 2$. A prolongable substitution has a fixed point \[\mathbf{a}=\lim_{n\to+\infty}\sigma^{n}(a)\] where the limit is taken under the natural metric and $\sigma^{n}$ represents the $n$ times composition $\sigma\circ\cdots\circ\sigma$. If, in addition, the substitution $\sigma$ is \emph{primitive}, i.e. there exists $k\in\mathbb{N}$ such that for all $a,\,b\in\mathcal{A}$, $b$ occurs in $\sigma^{k}(a)$, then $\mathbf{a}$ is uniformly recurrent (see \cite[Theorem 10.9.5]{AS03}).

A sequence is \emph{Sturmian} if for all $n\ge 1$, it has exactly $n+1$ different subwords of lenght $n$. A sequence $\mathbf{a}$ over the alphabet $\{a,b\}$ is \emph{balanced} if for any subwords $u, v$ of $\mathbf{a}$ with the same length, we have $||u|_a-|v|_a|\leq1$.  It is known that a sequence is Sturmian if and only if it is a non-eventually periodic balanced sequence over two letters. Note that, if $\mathbf{a}$ is Sturmian on the alphabet $\{a,b\}$, then exactly one of the words $aa, bb$ is not a subword of $\mathbf{a}$. We say that a Sturmian sequence $\mathbf{a}$ is of \emph{type} $a$ if $aa$ is a subword of $\mathbf{a}$. For more details on Sturmian sequences, see for example  \cite[Chapter 6]{1794}.

\subsection{Jacobi continued fraction}
Let $\mathbb{F}$ be a field and $x$ be an indeterminate. For a non-zero formal power series $f(x)=\sum_{n\ge 0}a_nx^n$, its order $\|f\|$ is the minimal $n$ such that $a_n\neq 0$; the order of zero series is $+\infty$. Then $\mathbb{F}[[x]]$ is the ring of formal power series with the metric $d(f,g)=2^{-\|f-g\|}$ for $f,g\in\mathbb{F}[[x]]$. Let $\mathbf{u}=u_1u_2\dots$ and $\mathbf{v}=v_0v_1v_2\dots$ be two sequences of elements in $\mathbb{F}$, where $v_i\neq 0$ for all $i$. Consider the finite continued fraction 
\[\mathbf{J}_n:=\mathbf{J}\begin{pmatrix}
	v_0 & v_1 & \dots & v_{n-1}\\
	u_1 & u_2 & \dots & u_n
	\end{pmatrix}:= \cfrac{v_0}{1+u_1x-\cfrac{v_1x^{2}}{1+u_2x-\cfrac{\ddots}{1+u_{n-1}x-\cfrac{v_{n-1}x^2}{1+u_{n}x}}}}\in \mathbb{F}[[x]].\]
The sequence $(\mathbf{J}_n)$ always converges in $\mathbb{F}[[x]]$; see for example \cite[Theorem 1]{Fla80}. Its limit defines the infinite continued fraction 
\[\mathbf{J}\begin{pmatrix}
	v_0 & v_1 & v_2 & \dots\\
	u_1 & u_2 & u_3 & \dots
	\end{pmatrix}:=
	\cfrac{v_0}{1+u_1x-\cfrac{v_1x^{2}}{1+u_2x-\cfrac{v_2x^2}{1+u_3x-\cfrac{v_3x^2}{\ddots}}}}\]
which is called the \emph{Jacobi continued fraction} (or \emph{$J$-fraction}) attached to $(\mathbf{u},\mathbf{v})$. The finite continued fraction $\mathbf{J}_n$ is called the  $n$-th \emph{approximant} of 
$\mathbf{J}\begin{pmatrix}
	v_0 & v_1 & v_2 & \dots\\
	u_1 & u_2 & u_3 & \dots
	\end{pmatrix}$.  

The following basic properties of $J$-fractions are obtained by Heilermann \cite{Heilermann1846}; see also \cite[Chapter IX \& XI]{Wa48}. 
A formal power series $f(x)$ yields a $J$-fraction expansion if its Hankel determinants are all non-vanishing, i.e. for all $n\geq 1$, $H_n(f)\neq 0$. A closer relation between the Hankel determinants and $J$-fraction expansion of a formal power series is the following: if
\[f(x)=\mathbf{J}\begin{pmatrix}
v_0 & v_1 & v_2 & \dots\\
u_1 & u_2 & u_3 & \dots
\end{pmatrix},\]
then for all $n\geq 1$,
\[H_n(f)=v_0^{n}v_1^{n-1}v_{2}^{n-2}\cdots v_{n-2}^{2}v_{n-1}.\]
Moreover, the \emph{order} of the $n$-th approximant is $2n+1$, namely,
\[f(x)-\mathbf{J}\begin{pmatrix}
v_0 & v_1 & \dots & v_{n-1}\\
u_1 & u_2 & \dots & u_n
\end{pmatrix} = O(z^{2n+1}).\]
The power series expansion (in ascending order) of $f(x)$ and its $n$-th approximant agree on just the first $2n$ terms.

\section{Criteria for apwenian sequences}\label{sec:cri} 
In this section, we first prove a criterion for 0-1 sequences to be apwenian. Then by studying the relation between 0-1 sequences and $\pm 1$ sequences, we transfer the 0-1 criterion to the $\pm 1$ criterion.

\subsection{0-1 Criterion}
Let $\mathbf{c}=c_0c_1c_2\cdots\in\{0,1\}^{\infty}$. Recall that $\mathbf{c}$ is apwenian if for all $n\geq 1$, its Hankel determinants
\[H_{n}(\mathbf{c})=\left|\begin{matrix}
c_{0} & c_{1} & \cdots & c_{n-1}\\
c_{1} & c_{2} & \cdots & c_{n}\\
\vdots & \vdots & \ddots & \vdots\\
c_{n-1} &  c_{n} & \cdots & c_{2n-2}
\end{matrix}\right|\equiv 1\pmod{2}.\]
Theorem \ref{thm:E} gives a sufficient and necessary condition for a 0-1 sequence to be apwenian.
To prove the result, we need some preparation.

\begin{lemma}\label{lem:1-2}
Let $f(x)=\sum_{i\geq 0}c_{i}x^{i}$ where $c_0=1$. Then, the condition \eqref{eq:01} is equivalent to
\begin{equation}\label{eq:02}
1+x^{2}f(x^2) \equiv  xf^{\text{odd}}(x) + f^{\text{even}}(x) \pmod{2},
\end{equation}
where $f^{\text{odd}}(x)=\sum_{i\geq 0}c_{2i+1}x^{2i+1}$ and $f^{\text{even}}(x)=\sum_{i\geq 0}c_{2i}x^{2i}$.
\end{lemma}
\begin{proof}
Let $\mathbf{b}=b_0b_1b_2\cdots\in\{0,1\}^{\infty}$ be a sequence satisfying \eqref{eq:01}. Write its generating function by $g(x)=\sum_{i\geq 0}b_ix^{i}$ and let   $g^{\text{odd}}(x)=\sum_{i\geq 0}b_{2i+1}x^{2i+1}$ and $g^{\text{even}}(x)=\sum_{i\geq 0}b_{2i}x^{2i}$. 

For any $g$ one has $b_{0}+x^{2}g(x^{2})=b_0 + x^{2}\sum_{i\geq 0}b_{i}x^{2i}$ and 
\begin{align*}
	xg^{\text{odd}}(x)+g^{\text{even}}(x) &  = \sum_{i\geq 0}b_{2i+1}x^{2i+2}+\left(b_{0}+\sum_{i\geq 0}b_{2i+2}x^{2i+2}\right)\\
	& = b_0 + \sum_{i\geq 0}(b_{2i+1}+b_{2i+2})x^{2i+2}.
\end{align*}
Thus $b_{0}+x^{2}g(x^{2})\equiv xg^{\text{odd}}(x)+g^{\text{even}}(x)\pmod{2}$ if and only if $b_i\equiv b_{2i+1}+b_{2i+2}\pmod{2}$ for all $i\ge 0$.
\end{proof}

\begin{lemma}\label{lem:gf}
	Let $u_1\in \{0,1\}$ and $g(x)=\sum_{i\geq 0}b_ix^{i}$ where $b_0=1$ and $b_i\in\{0,1\}$ for all $i\geq 1$. Define \[f(x):=\cfrac{1}{1+u_1x-x^2g(x)}.\] Then $f(x)$ satisfies \eqref{eq:02} if and only if $g(x)$ satisfies \eqref{eq:02}.
\end{lemma}

\begin{proof}
Write $h(x) :=  f(x)(1+u_1x-x^2g(x))$. Then $h(x)=1$ means that $h^{\text{odd}}(x)= 0$ and $h^{\text{even}}(x)= 1$.  Note that
\begin{align*}
h(x) = &\,  f(x)(1+u_1x-x^2g(x))\\
 = &\, (f^{\text{odd}}(x)+f^{\text{even}}(x))\Big[\big(u_1x-x^2g^{\text{odd}}(x)\big)+\big(1-x^2g^{\text{even}}(x)\big)\Big].
\end{align*}
So
\begin{equation}
\left\{
\begin{aligned}
0 & = h^{\text{odd}}(x) = f^{\text{odd}}(x)\big(1-x^2g^{\text{even}}(x)\big)+f^{\text{even}}(x)\big(u_1x-x^2g^{\text{odd}}(x)\big),\\
1 & = h^{\text{even}}(x) = f^{\text{even}}(x)\big(1-x^2g^{\text{even}}(x)\big)+f^{\text{odd}}(x)\big(u_1x-x^2g^{\text{odd}}(x)\big).
\end{aligned}\right.\label{eq:h-fg}
\end{equation}

Solving the linear system \eqref{eq:h-fg} in variables $f^{\text{odd}}$ and $f^{\text{even}}$, one obtains
\[f^{\text{even}}(x) = f(x)f(-x)\left(1-x^2g^{\text{even}}(x)\right)\text{ and }
f^{\text{odd}}(x) = f(x)f(-x)\left(x^2g^{\text{odd}}(x)-u_1x\right).\]
Consequently,
\begin{equation}\label{f-equation}
  f^{\text{odd}}(x) \equiv (u_1x+x^{2}g^{\text{odd}}(x))f(x^2) \text{ and } f^{\text{even}}(x) \equiv (1+x^{2}g^{\text{even}}(x))f(x^2)\pmod{2}.
\end{equation}
If $g(x)$ satisfies \eqref{eq:02}, then
\begin{align*}
xf^{\text{odd}}(x)+f^{\text{even}}(x) & \equiv x(u_1x+x^{2}g^{\text{odd}}(x))f(x^2)+(1+x^{2}g^{\text{even}}(x))f(x^2)\\
& \equiv f(x^{2})\Big(u_1x^{2}+x^{3}g^{\text{odd}}(x)+1+x^{2}g^{\text{even}}(x)\Big)\\
& \equiv f(x^{2})\Big(1+u_1x^{2}+x^{2}\big(1+x^{2}g(x^{2})\big)\Big)\qquad\text{by }\eqref{eq:02} \text{ for }g(x)\\
& \equiv f(x^{2})\Big(\frac{1}{f(x^{2})}+x^{2}\Big)\\
& \equiv 1+x^{2}f(x^{2})\pmod{2}
\end{align*}
which is \eqref{eq:02} for $f(x)$. Conversely, if $f(x)$ satisfies \eqref{eq:02}, then noticing that $\frac{1}{f(x^2)}=1+u_1x^2-x^4g(x^2)$, we have
\begin{align*}
1+(1+u_1)x^2-x^4g(x^2)& \equiv \frac{1}{f(x^2)}\Big(1+x^2f(x^2)\Big)\\
& \equiv \frac{1}{f(x^2)}\Big(xf^{\text{odd}}(x)+f^{\text{even}}(x)\Big) \pmod{2}\qquad\text{by }\eqref{eq:02} \text{ for }f(x) \\
& \equiv u_1x^{2}+x^{3}g^{\text{odd}}(x)+1+x^{2}g^{\text{even}}(x)\pmod{2}.\qquad \text{by }\eqref{f-equation}
\end{align*}
Therefore,
\begin{align*}
x^2-x^4g(x^2) & \equiv x^2\left(xg^{\text{odd}}(x)+g^{\text{even}}(x)\right)\pmod{2}
\end{align*}
which implies that $g(x)$ satisfies \eqref{eq:02}.
\end{proof}

\begin{lemma}\label{lem:tm}
Let $(t_n)_{n\geq 0}$ be the $0$-$1$ Thue-Morse sequence given by $t_0=1$, $t_{2n}=t_n$ and $t_{2n+1}\equiv1+t_n~ (\mathrm{mod}~2)$ for all $n\geq 0$. Then $g(x)=\sum_{i\geq 0}(t_{n}+t_{n+2})x^{n}$ satisfies \eqref{eq:02}.
\end{lemma}
\begin{proof}
Write $h(x):=\sum_{i\geq 0}t_ix^i$. Then $g(x)=h(x)+\frac{1}{x^2}\big(h(x)-1\big)$. Note that
\begin{align*}
g^{\text{odd}}(x) & = \sum_{i\geq 0}(t_{2i+1}+t_{2i+3})x^{2i+1} \\
	&\equiv \sum_{i\geq 0}(t_{i}+t_{i+1})x^{2i+1}\\
& \equiv xh(x^2)+\frac{1}{x}\Big(h(x^2)-1\Big) \pmod{2}
\end{align*}
and
\begin{align*}
g^{\text{even}}(x) & = \sum_{i\geq 0}(t_{2i}+t_{2i+2})x^{2i} \\
	&\equiv \sum_{i\geq 0}(t_{i}+t_{i+1})x^{2i}\\
& \equiv h(x^2)+\frac{1}{x^2}\Big(h(x^2)-1\Big)\pmod{2}.
\end{align*}
Therefore,
\begin{align*}
xg^{\text{odd}}(x)+g^{\text{even}}(x) & \equiv x^2h(x^2)+h(x^2)-1+h(x^2)+\frac{1}{x^2}\Big(h(x^2)-1\Big)\\
& \equiv 1+ x^2h(x^2)+\frac{1}{x^2}\Big(h(x^2)-1\Big)\\
& \equiv 1+x^2g(x^2) \pmod{2}.
\end{align*}
So $g(x)$ satisfies \eqref{eq:02}.
\end{proof}

Now we are ready to prove Theorem \ref{thm:E}.
\begin{proof}[Proof of Theorem \ref{thm:E}]
{\it The `only if' part.} Let $f(x)=\sum_{i\geq 0}c_{i}x^{i}$. We can check directly that $H_{1}(f)\equiv 1~ (\mathrm{mod}~2)$ implies $c_0\equiv 1~(\mathrm{mod}~2)$, and $H_1(f)\equiv H_2(f)\equiv 1~ (\mathrm{mod}~2)$ implies $c_0\equiv 1 ~(\mathrm{mod}~2)$ and $c_1+c_2\equiv c_0~(\mathrm{mod}~2)$. For any $k\geq 2$, suppose $H_{n}(f)\equiv 1~(\mathrm{mod}~2)$ for all $n\in [1,k]$. Then $f$ yields a $J$-fraction expansion
\begin{equation}\label{eq:j-f}
f(x)=\cfrac{1}{1+u_1x-\cfrac{x^{2}}{1+u_2x-\cfrac{x^{2}}{\qquad\cfrac{\ddots}{1+u_{k-1}x-x^{2}f_{k}(x)}}}}.
\end{equation}
Replace $f_{k}(x)$ in \eqref{eq:j-f} by $g(x)$ in Lemma \ref{lem:tm}, and denote the resulting series by $f^{\prime}(x)=\sum_{i\geq 0}c^{\prime}_{i}x^{i}$. Since $f(x)$ and $f^{\prime}(x)$ have the same $(k-1)$-th approximant, and the order of the $(k-1)$-th approximant is $2(k-1)+1$, we have $c_i=c_i^{\prime}$ for $i=0,1,\dots,2k-2$. It follows from Lemmas \ref{lem:tm} and \ref{lem:gf} that $f^{\prime}(x)$ satisfies \eqref{eq:02}. Then by Lemma \ref{lem:1-2},  $c_0=1$ and $c_i\equiv c_{2i+1}+c_{2i+2}~(\mathrm{mod}~2)$ holds for all $i\in [0,k-2]$.

{\it The `if' part.} For the converse, suppose that $c_0=1$ and $c_i\equiv c_{2i+1}+c_{2i+2}~(\mathrm{mod}~2)$ holds for all $i\geq 0$. It follows from $c_0\equiv 1~(\mathrm{mod}~2)$ that $H_1(f)\equiv 1~(\mathrm{mod}~2)$. Consequently, $f(x)$ yields a $J$-fraction expansion
\[f(x)=\cfrac{1}{1+u_1x-x^2f_1(x)}.\]
Since $f(x)$ satisfies \eqref{eq:02}, by Lemma \ref{lem:gf}, $f_1(x)$ also satisfies \eqref{eq:02}. So $f_1(x)$ has a $J$-fraction expansion
\[f_1(x)=\cfrac{1}{1+u_2x-x^2f_2(x)}.\]
Now $f_1(x)$ satisfies \eqref{eq:02}. This yields that $H_1(f_1)\equiv 1~(\mathrm{mod}~2)$ and $f_2(x)$ satisfies \eqref{eq:02}. In summary,
\[f(x)=\cfrac{1}{1+u_1x-\cfrac{x^2}{1+u_2x-x^2f_2(x)}}\]
and $f_2(x)$ satisfies \eqref{eq:02}.
Repeating the previous argument, we find the $J$-fraction expansion of~$f(x)$: \[\mathbf{J}\begin{pmatrix}
1 & 1 & 1 & \cdots\\
u_1 & u_2 & u_3 & \cdots
\end{pmatrix}\] where $u_i\in\{0,1\}$ for all $i\geq 1$. Hence $H_i(f)\equiv 1~(\mathrm{mod}~2)$ for all $i\geq 1$.
\end{proof}

\subsection{$\pm 1$ Criterion} Let $\mathbf{d}=d_0d_1d_2\cdots \in\{-1,1\}^{\infty}$. Recall that $\mathbf{d}$ is apwenian if for all $n\geq 1$, \[H_n(\mathbf{d})/2^{n-1}\equiv 1\quad (\mathrm{mod}~2).\]
Our criterion for a $\pm 1$ sequence to be apwenian is stated in Theorem \ref{thm:A}.
Before we prove Theorem \ref{thm:A}, we need a relation between $0$-$1$ apwenian sequences and $\pm 1$ apwenian sequences.
\begin{lemma}\label{lem:01+-1}
Let $\mathbf{d}=d_0d_1d_2\cdots\in\{-1,1\}^{\infty}$ and let $\mathbf{c}=c_0c_1c_2\cdots\in\{0,1\}^{\infty}$ where $c_i\equiv\frac{d_{i}-d_{i+2}}{2}\,\, (\mathrm{mod}~2)$ for all $i\geq 0$. Then $\mathbf{d}$ is $\pm1$ apwenian if and only if $\mathbf{c}$ is $0$-$1$ apwenian.
\end{lemma}
\begin{proof} Let $b_i=\frac{1-d_i}{2}$ for all $i\geq 0$. Then
\begin{align*}
H_n(\mathbf{d})
	&= \left|\begin{matrix}
1-2b_0 & 1-2b_1 & \cdots & 1-2b_{n-1}\\
1-2b_1 & 1-2b_2 & \cdots & 1-2b_n\\
\vdots & \vdots & \ddots & \vdots\\
1-2b_{n-1} & 1-2b_n & \cdots & 1-2b_{2n-2}
\end{matrix}\right|\\
& = 2^{n-1}\left|\begin{matrix}
1-2b_0 & b_0-b_1 & \cdots & b_{n-2}-b_{n-1}\\
1-2b_1 & b_1-b_2 & \cdots & b_{n-1}-b_n\\
\vdots & \vdots & \ddots & \vdots\\
1-2b_{n-1} & b_{n-1}-b_n & \cdots & b_{2n-3}-b_{2n-2}
\end{matrix}\right|\\
& = 2^{n-1}\left|\begin{matrix}
1 & b_0-b_1 & \cdots & b_{n-2}-b_{n-1}\\
1 & b_1-b_2 & \cdots & b_{n-1}-b_n\\
\vdots & \vdots & \ddots & \vdots\\
1 & b_{n-1}-b_n & \cdots & b_{2n-3}-b_{2n-2}
\end{matrix}\right| -
2^n\left|\begin{matrix}
b_0 & b_0-b_1 & \cdots & b_{n-2}-b_{n-1}\\
b_1 & b_1-b_2 & \cdots & b_{n-1}-b_n\\
\vdots & \vdots & \ddots & \vdots\\
b_{n-1} & b_{n-1}-b_n & \cdots & b_{2n-3}-b_{2n-2}
\end{matrix}\right|\\
& =  2^{n-1}\left|\begin{matrix}
1 & b_0-b_1 & \cdots & b_{n-2}-b_{n-1}\\
2 & b_0-b_2 & \cdots & b_{n-2}-b_n\\
\vdots & \vdots & \ddots & \vdots\\
2 & b_{n-2}-b_n & \cdots & b_{2n-4}-b_{2n-2}
\end{matrix}\right| -
2^n\left|\begin{matrix}
b_0 & -b_1 & \cdots & -b_{n-1}\\
b_1 & -b_2 & \cdots & -b_n\\
\vdots & \vdots & \ddots & \vdots\\
b_{n-1} & -b_n & \cdots & -b_{2n-2}
\end{matrix}\right|.
\end{align*}
Dividing both sides by $2^{n-1}$, we have
\begin{align*}
\frac{H_n(\mathbf{d})}{2^{n-1}} & =  \left|\begin{matrix}
1 & b_0-b_1 & \cdots & b_{n-2}-b_{n-1}\\
2 & b_0-b_2 & \cdots & b_{n-2}-b_n\\
\vdots & \vdots & \ddots & \vdots\\
2 & b_{n-2}-b_n & \cdots & b_{2n-4}-b_{2n-2}
\end{matrix}\right| -
2 \left|\begin{matrix}
b_0 & -b_1 & \cdots & -b_{n-1}\\
b_1 & -b_2 & \cdots & -b_n\\
\vdots & \vdots & \ddots & \vdots\\
b_{n-1} & -b_n & \cdots & -b_{2n-2}
\end{matrix}\right|\\
& \equiv
\left|\begin{matrix}
1 & b_0-b_1 & \cdots & b_{n-2}-b_{n-1}\\
0 & b_0-b_2 & \cdots & b_{n-2}-b_n\\
\vdots & \vdots & \ddots & \vdots\\
0 & b_{n-2}-b_n & \cdots & b_{2n-4}-b_{2n-2}
\end{matrix}\right| \pmod{2}\\
	&\equiv \left|\begin{matrix}
c_0 & \cdots & c_{n-2}\\
 \vdots & \ddots & \vdots\\
c_{n-2} & \cdots & c_{2n-4}
\end{matrix}\right| \pmod{2}\\
	&\equiv H_{n-1}(\mathbf{c}) \pmod{2}. \qedhere
\end{align*}
\end{proof}

\goodbreak

Lemma \ref{lem:01+-1} allows us to transfer Theorem \ref{thm:E} to be a criterion for $\pm 1$ apwenian sequences.

\begin{proof}[Proof of Theorem \ref{thm:A}]
By Theorem \ref{thm:E} and Lemma \ref{lem:01+-1},
the $\pm1$ sequence	$\mathbf{d}$ is apwenian if and only if $d_0=-d_2$ and
\begin{equation}\label{eq:cor1-1}
\forall\, n\geq 0,\quad\frac{d_n-d_{n+2}}{2}\equiv \frac{d_{2n+1}-d_{2n+3}}{2}+\frac{d_{2n+2}-d_{2n+4}}{2}\pmod{2}.
\end{equation}
Adding up the equation \eqref{eq:cor1-1}, we obtain its equivalent form: $\forall\, n\geq 0,$
\[\sum_{i=0}^{n}\frac{d_i-d_{i+2}}{2}\equiv \sum_{i=0}^{n}\left(\frac{d_{2i+1}-d_{2i+3}}{2}+\frac{d_{2i+2}-d_{2i+4}}{2}\right)\pmod{2}\]
which reduces to
\begin{equation}\label{eq:cor1-2}
\forall\, n\geq 0,\quad \frac{d_0+d_1}{2}-\frac{d_{n+1}+d_{n+2}}{2}\equiv \frac{d_1+d_2}{2}-\frac{d_{2n+3}+d_{2n+4}}{2}\pmod{2}.
\end{equation}
	So \eqref{eq:cor1-1} is equivalent to \eqref{eq:cor1-2}. It follows from the initial condition $d_0=-d_2$ that \eqref{eq:cor1-2} is exactly~\eqref{eq:-01} for $n\geq 1$. Since $d_0=-d_2$ is just \eqref{eq:-01} for $n=0$, we conclude that   \eqref{eq:cor1-2} and  \eqref{eq:-01} are equivalent.
This completes the proof.
\end{proof}

\section{The 0-1 apwenian sequences generated by type I substitutions}\label{sec:01} 
In this section, we focus on sequences generated by substitutions of constant length $p\geq 2$ on $\{0,1\}$ such that
\begin{equation*}
\sigma: 1\mapsto 1w_1\cdots w_{p-1},\quad 0\mapsto v_0v_1\cdots v_{p-1},
\end{equation*}
where $w_i, v_i\in\{0,1\}$ for all $i$. These substitutions are called of {\it type I}.
Let $\mathbf{c}=c_0c_1c_2\cdots$ be the fixed point of $\sigma$ starting with $1$.

\begin{remark*}
If $\sigma$ is a substitution such that $\sigma(1)$ is not starting with $1$, then either the substitution does not have a fixed point or the fixed point is starting with $0$ and it is not apwenian.
\end{remark*}

\begin{example}\label{eg:1}
Let $\sigma$ be the substitution $1\mapsto 10$ and $0\mapsto 11$. Its fixed point
\[\mathbf{c}=101110101011101\cdots\]
is the well known \emph{period-doubling} sequence; see for example \cite{D20}. The sequence $\mathbf{c}$ satisfies the recurrence relations $c_{2n}=1$ and $c_{2n+1}=1-c_n$ ($n\geq 0)$. Hence, by Theorem \ref{thm:E}, $\mathbf{c}$ is apwenian.
\end{example}

It is interesting to see that substitutions of constant length actually give only one $0
$-$1$ apwenian sequence, namely, the period-doubling sequence.

\begin{theorem}
Let $\mathbf{c}\in\{0,1\}^{\infty}$ be the fixed point of a type I substitution of length $p$ such that $\mathbf{c}$ is starting with $1$. Then
$\mathbf{c}$ is $0$-$1$ apwenian if and only if  $\mathbf{c}$ is the period-doubling sequence.
\end{theorem}
\begin{proof}
Let $\mathbf{c}$ be the fixed point of $\sigma$. Then it is also the fixed point of  $\sigma^m$ for all $m\geq1$.
Hence, we always assume that the length of $\sigma$ is $p$ with $p\geq3$.
By Example \ref{eg:1}, we only have to prove the necessity.
Since $c_0c_1c_2\cdots=\mathbf{c}=\sigma(\mathbf{c})=\sigma(c_0)\sigma(c_1)\sigma(c_2)\cdots$, we have for all $i\geq 0$,
\begin{equation}\label{eq:sub-1}
\sigma(c_i)=c_{ip}c_{ip+1}\cdots c_{(i+1)p-1}.
\end{equation}

Suppose that $\mathbf{c}$ satisfies \eqref{eq:01}. Then $c_0=1$ and $c_1+c_2\equiv 1 \, (\mathrm{mod}~2)$. There are only two cases to be considered.

{\it Case 1:} $c_1=1$ and $c_2=0$. By \eqref{eq:sub-1}, $\sigma(c_0)=\sigma(c_1)=\sigma(1)$ implies \[c_{p-1}=c_{2p-1} \text{ and }c_p=c_0=1.\] By \eqref{eq:01}, $c_{p-1}\equiv c_{2p-1}+c_{2p} ~(\mathrm{mod}~ 2)$ and $c_{2p}\equiv c_{4p+1}+c_{4p+2} ~(\mathrm{mod}~2)$. Hence, $0\equiv c_{2p}\equiv c_{4p+1}+c_{4p+2}~ (\mathrm{mod}~2)$. Since $\sigma(1)=110\cdots c_{p-1}$ and $\sigma(c_4)=c_{4p}c_{4p+1}c_{4p+2}\cdots c_{5p-1}$, we have $c_4=0$. It follows from \eqref{eq:sub-1} and $c_2=c_4=0$ that
\[c_{2p+1}c_{2p+2}=c_{4p+1}c_{4p+2}.\]
So $c_p \equiv c_{2p+1}+c_{2p+2}=c_{4p+1}+c_{4p+2}\equiv c_{2p} =0 ~(\mathrm{mod}~2)$ which is a contradiction.

\smallskip

{\it Case 2:} $c_1=0$ and $c_2=1$. By \eqref{eq:sub-1}, $c_{2p}=c_{2p+2}=1$ and $c_{2p+1}=0.$ Then, by \eqref{eq:01}, $c_{p}\equiv c_{2p+1}+c_{2p+2}=1~ (\mathrm{mod}~ 2)$.  This implies that both $\sigma(0)$ and $\sigma(1)$ are starting with $1$. By \eqref{eq:sub-1}, we have for all $i\geq 0$, 
\begin{equation}\label{eq:in}
c_{ip}=1.
\end{equation}
Using \eqref{eq:01} and \eqref{eq:in}, for $i\geq 0$, we have
\[c_{(i+1)p-1}\equiv c_{2(i+1)p-1}+c_{2(i+1)p} =c_{2(i+1)p-1}+1 \pmod{2}\]
which implies that the last letter of $\sigma(c_{i})$ and $\sigma(c_{2i+1})$ are different by \eqref{eq:sub-1}. So $c_{i}\neq c_{2i+1}$. In the other words, for all $i\geq 0$,
\begin{equation}\label{eq:2i+1}
c_{2i+1}\equiv c_{i} +1 \pmod{2}.
\end{equation}
It follows from \eqref{eq:01} and \eqref{eq:2i+1} that for all $i\geq 0$,
\begin{equation}\label{eq:2i}
c_{2i+2} \equiv c_i + c_{2i+1} \equiv c_{i}+ c_{i} + 1 \equiv 1\pmod{2}.
\end{equation}
The recurrence relations \eqref{eq:2i+1} and \eqref{eq:2i} with the initial value $c_0=1$ show that $\mathbf{c}$ is the period doubling sequence.
\end{proof}

\section{The \texorpdfstring{$\pm 1$}{+-1}  apwenian sequences generated by type II substitutions}\label{sec:1-1}  
In this section, we focus on a particular class of substitutions of constant length on the alphabet $\{-1,1\}$. Based on Theorem \ref{thm:A}, we give a detailed criterion (Theorem \ref{thm:t2}) on the substitution itself to tell whether its fixed point is apwenian or not.  This detailed criterion not only allows us to give the exact formula for the number of apwenian sequences (see Propositions \ref{prop:1} and \ref{prop:21}), but also provides a way to write down these apwenian sequences (see Remark \ref{re:p=9}). We call this particular class of substitutions of {\it type II}.
\begin{definition}[Type II substitution]  We say a substitution $\sigma$ on $\mathcal{A}=\{-1,1\}$ is of {\it type II} if it satisfies
\begin{equation*}
1\mapsto v_0v_1\cdots v_{p-1},\quad -1\mapsto \bar{v}_0\bar{v}_1\cdots \bar{v}_{p-1}
\end{equation*}
where $p\geq 2$, $v_0=1$ and $\bar{v}_i:=-v_i$ for $i=0,1,\dots,p-1$. The \emph{length} of the substitution $\sigma$ is $p$.
\end{definition}

Let $\mathbf{d}=d_0d_1d_2\cdots\in\{-1,1\}^{+\infty}$ be the fixed point of a type II substitution starting with $1$. It follows from $\mathbf{d}=\sigma(\mathbf{d})$ that for all $n\geq 0$, $\sigma(d_n)=d_{np}d_{np+1}\cdots d_{np+p-1}$. Then for all $n\geq 0$ and $j=0,1,\dots,p-1$,
\begin{equation}\label{eq:rec}
d_{np+j}=v_jd_n.
\end{equation}
Here $v_jd_n$ represents the product of integers $v_j$ and $d_n$.
Moreover, from \eqref{eq:rec}, we find that the generating function $f_p(x)=\sum_{n=0}^{+\infty}d_n x^n$ satisfies
\begin{equation}\label{eq:fp}
f_p(x)=(v_0+v_1x+\cdots+v_{p-1}x^{p-1})f_p(x^p)
\end{equation}
where $v_i\in\{-1,1\}$.
Now we introduce our criterion for $\pm 1$ apwenian sequences generated by type II substitutions.
\begin{theorem}\label{thm:t2}
Let $\mathbf{d}\in\{-1,1\}^{+\infty}$ be the fixed point of a type II substitution of length $p\geq 2$ and $\mathbf{d}$ is starting with $1$. For $m\geq p$, define $v_m:=v_j$ where $m=np+j$ with $n\geq 0$ and $j=0,1,\dots,p-1$.

$(1)$ When $p$ is odd, $\mathbf{d}$ is $\pm 1$ apwenian if and only if
\begin{equation}
\label{odd}
\frac{v_j+v_{j+1}-v_{2j+1}-v_{2j+2}}{2}  \equiv 1\pmod{2},\quad 0\leq j\leq p-2.
\end{equation}

$(2)$ When $p$ is even, $\mathbf{d}$ is $\pm 1$ apwenian if and only if  $\mathbf{d}$ is the Thue-Morse sequence.
\end{theorem}
\begin{remark}
Let $p$ be even and $N_p$ be the number of apwenian series satisfying \eqref{eq:fp}.  Then, Theorem \ref{thm:t2} implies that 
$$N_p=\begin{cases}
    1,  & \text{if $p=2^k$ for some $k\geq1,$ }\\
     0, & \text{otherwise}.
\end{cases}$$
 \end{remark}

\begin{proof}[Proof of Theorem \ref{thm:t2}]
If $\mathbf{d}\in\{-1,1\}^{+\infty}$ is the fixed point of a type II substitution of length $p\geq 2$ and $\mathbf{d}$ is starting with $1$, then $\mathbf{d}$ satisfies \eqref{eq:rec}. To apply Theorem \ref{thm:A}, we need to check the relation between \eqref{eq:-01} and \eqref{odd}.

When $0\leq j\leq \frac{p-3}{2}$, we have $0\leq 2j+1<2j+2\leq p-1$. For all $n\geq 0$,
\begin{align}
&\frac{d_{np+j}+d_{np+j+1}-d_{2(np+j)+1}-d_{2(np+j)+2}}{2} \notag\\
=\, & \frac{d_{np+j}+d_{np+j+1}-d_{2np+2j+1}-d_{2np+2j+2}}{2}   \notag \\
=\, &  \frac{v_jd_{n}+v_{j+1}d_{n}-v_{2j+1}d_{2n}-v_{2j+2}d_{2n}}{2} \qquad\text{\qquad\qquad by \eqref{eq:rec}}\notag\\
=\, &  \frac{v_j+v_{j+1}}{2}d_{n}-\frac{v_{2j+1}+v_{2j+2}}{2}d_{2n} \notag\\
\equiv\, &  \frac{v_j+v_{j+1}}{2}-\frac{v_{2j+1}+v_{2j+2}}{2} \pmod{2}.\label{1}
\end{align}
When $\frac{p-1}{2}\leq j\leq p-2$, we have $p \leq 2j+1<2j+2\leq 2p-2$. For all $n\geq 0$,
\begin{align}
& \frac{d_{np+j}+d_{np+j+1}-d_{2(np+j)+1}-d_{2(np+j)+2}}{2}\notag \\
=\, &  \frac{d_{np+j}+d_{np+j+1}-d_{(2n+1)p+2j+1-p}-d_{(2n+1)p+2j+2-p}}{2}  \notag\\
=\, & \frac{v_jd_{n}+v_{j+1}d_{n}-v_{2j+1-p}d_{2n+1}-v_{2j+2-p}d_{2n+1}}{2}\qquad\text{by \eqref{eq:rec}}\notag\\
=\, & \frac{v_j+v_{j+1}}{2}d_{n}-\frac{v_{2j+1-p}+v_{2j+2-p}}{2}d_{2n+1}\notag\\
\equiv\, &  \frac{v_j+v_{j+1}}{2}-\frac{v_{2j+1}+v_{2j+2}}{2} \pmod{2}.\label{2}
\end{align}
When $j=p-1$, we have for all $n\geq 0$,
\begin{align}
& \frac{d_{np+j}+d_{np+j+1}-d_{2(np+j)+1}-d_{2(np+j)+2}}{2} \notag\\
=\, &  \frac{d_{np+p-1}+d_{np+p}-d_{2np+2p-1}-d_{2np+2p}}{2}  \notag\\
=\, &  \frac{v_{p-1}d_{n}+v_{0}d_{n+1}-v_{p-1}d_{2n+1}-v_{0}d_{2n+2}}{2}\qquad\text{\qquad by \eqref{eq:rec}}\notag\\
=\, &  v_{p-1}\frac{d_n-d_{2n+1}}{2} + v_{0}\frac{d_{n+1}-d_{2n+2}}{2}\notag\\
\equiv\, &  \frac{d_n-d_{2n+1}}{2}+\frac{d_{n+1}-d_{2n+2}}{2} \pmod{2}.\label{3}
\end{align}

$(1)$ If $p$ is odd, then $\frac{p-3}{2}$ and $\frac{p-1}{2}$ are two consecutive integers. If $\mathbf{d}$ is $\pm 1$ apwenian, then by Theorem \ref{thm:A}, we obtain \eqref{odd} from \eqref{1} and \eqref{2}. Conversely, if \eqref{odd} holds, then \eqref{1}, \eqref{2} and \eqref{3} yield \eqref{eq:-01} by induction on $n$. It follows from Theorem \ref{thm:A} that $\mathbf{d}$ is $\pm 1$ apwenian.

\medskip

$(2)$ Suppose $p$ is even and $p=2q$ with $q\geq 1$. If $\mathbf{d}$ is $\pm 1$ apwenian, then by  Theorem \ref{thm:A}, \eqref{1} and \eqref{2}, we have
$$\frac{v_j+v_{j+1}-v_{2j+1}-v_{2j+2}}{2}  \equiv 1  ~(\mathrm{mod}~2),\quad 0\leq j\leq p-2 \text{ and } j\neq q-1.$$
 Hence, $$\sum_{\substack{0\leq j\leq p-2   \\    j\neq q-1}
}\frac{v_j+v_{j+1}-v_{2j+1}-v_{2j+2}}{2}\equiv p-2\equiv 0 \pmod{2}.$$
It follows that
\begin{equation}
\frac{v_{q-1}+v_{q}}{2}\equiv \frac{v_{0}+v_{p-1}}{2} \pmod{2}.\label{4}
\end{equation}
Thus,
\begin{align}
&\quad~\frac{d_{np+q-1}+d_{np+q}-d_{2(np+q-1)+1}-d_{2(np+q-1)+2}}{2}\notag\\
& =  \frac{d_{np+q-1}+d_{np+q}-d_{2np+p-1}-d_{(2n+1)p}}{2}  \notag\\
& =  \frac{v_{q-1}d_{n}+v_{q}d_{n}-v_{p-1}d_{2n}-v_{0}d_{2n+1}}{2}\qquad\text{\qquad\qquad by \eqref{eq:rec}}\notag\\
& =  \frac{v_{q-1}+v_{q}}{2}d_{n}-\frac{v_{p-1}d_{2n}+v_{0}d_{2n+1}}{2}\notag\\
& \equiv  \frac{v_{0}+v_{p-1}}{2}d_{n}-\frac{v_{p-1}d_{2n}+v_{0}d_{2n+1}}{2}\pmod{2}\qquad\text{by \eqref{4}}\notag\\
& =  \frac{d_{n}-d_{2n+1}}{2}v_{0}+\frac{d_{n}-d_{2n}}{2}v_{p-1}\notag\\
& \equiv  \frac{d_{n}-d_{2n+1}}{2}+\frac{d_{n}-d_{2n}}{2}\pmod{2}\notag\\
&\equiv 1+\frac{d_{2n}+d_{2n+1}}{2} \pmod{2}.  \label{5}
\end{align}
If $\mathbf{d}$ is $\pm 1$ apwenian, then by \eqref{eq:-01} and \eqref{5}, we have for all $n\geq 0$, $\frac{d_{2n}+d_{2n+1}}{2}\equiv 0 ~(\mathrm{mod}~2)$. This implies that
\begin{equation}
\forall~ n\geq 0, \quad d_{2n+1}=-d_{2n}.\label{even-1}
\end{equation}
By Theorem \ref{thm:A} and \eqref{even-1}, we obtain that for all $n\geq0$,
\[\frac{d_n+d_{n+1}-d_{2n+1}-d_{2n+2}}{2}\equiv \frac{d_n+d_{n+1}+d_{2n}-d_{2n+2}}{2} \equiv 1 \pmod{2}.\]
Hence,
\[\sum_{i=0}^{n-1} \frac{d_i+d_{i+1}+d_{2i}-d_{2i+2}}{2}\equiv n \pmod{2}\]
which reduces to
\begin{equation}
\forall~n\geq 0,\quad d_{2n}=d_n.\label{even-2}
\end{equation}
It follows from \eqref{even-1} and \eqref{even-2} that $\mathbf{d}$ is the Thue-Morse sequence.
\end{proof}

\subsection{The $\pm 1$ apwenian sequences when $p$ is odd}
Let $p\geq 3$ be an odd number and let $\sigma: 1\mapsto v_0v_1,\cdots v_{p-1},\, 0\mapsto \bar{v}_0\bar{v}_1\cdots\bar{v}_{p-1}$ be a type II substitution of length $p$. The mapping $j\mapsto 2j+1~(\mathrm{mod}~p)$ defined on $\mathbb{N}$  induces the permutation
\[\tau: \begin{pmatrix}
0& 1 &\cdots& \frac{p-3}{2}& \frac{p-1}{2}& \frac{p+1}{2}& \cdots& p-2\\
1& 3& \cdots& p-2& 0& 2& \cdots& p-3
\end{pmatrix}.\]
For $j=0,1,\dots,p-2$, write \[\delta(j):=\frac{v_j-v_{j+1}}{2}.\] Then $\delta(j)\in\{0,1\}$ and \eqref{odd} is equivalent to the linear system
\begin{equation}\label{odd-2}
\begin{bmatrix}
\delta(0)\\
\vdots\\
\delta(p-2)
\end{bmatrix}+\begin{bmatrix}
\delta(\tau(0))\\
\vdots\\
\delta(\tau(p-2))
\end{bmatrix}\equiv
\begin{bmatrix}
1\\ \vdots \\ 1
\end{bmatrix}\pmod{2}.
\end{equation}
Note that \eqref{odd-2} is a linear system on the $\mathbb{F}_2$-vector space $\mathbb{F}_2^{p-1}$. Then it has either no solution or $2^k$ solutions for some $k\geq 0$.

By Theorem \ref{thm:t2}, searching for $\pm 1$ apwenian sequences that are fixed points of type II substitutions turns out to be solving the linear system \eqref{odd-2}. The following result shows that the solutions of \eqref{odd-2} can be characterized by the cycle decomposition of the permutation $\tau$.

\begin{proposition}\label{prop:1}
	Let $p\geq 3$ be an odd number. 

	$(1)$ The linear system \eqref{odd-2} has no solution if and only if there is an odd cycle in the cycle decomposition of $\tau$.
	
	$(2)$ The linear system \eqref{odd-2} has $2^k$ solutions if and only if $\tau$ can be decomposed into $k$ cycles of even length and $\tau$ contains no cycles of odd length.
	\end{proposition}
	
	\begin{proof}
	Let $h$ be the number of cycles in the cycle decomposition of $\tau$. Write the cycle decomposition of $\tau$ as
	\begin{equation}\label{eq:cd-r2}
		\Bigl(x^{(1)}_1,\dots,x^{(1)}_{r_1}\Bigr)\Bigl(x^{(2)}_1,\dots,x^{(2)}_{r_2}\Bigl)\dots\Bigl(x^{(h)}_1,\dots,x^{(h)}_{r_h}\Bigr)
	\end{equation}
	where $r_i\geq 1$ ($i=1,\dots,h$) and $r_1+\dots+r_h=p-1$. We reorder the linear system \eqref{odd-2} into $h$ subsystems according to the cycles. For $\ell=1,2,\dots,h$, the $\ell$-th cycle yields the $\ell$-th subsystem 
	\begin{equation}\label{eq:subsystem-l}
		\left\{
		\begin{aligned}
			\delta(x^{(\ell)}_1) & + \delta(\tau(x^{(\ell)}_1)) \equiv 1,\\
			&\ \,\vdots\\
			\delta(x^{(\ell)}_{r_{\ell}}) & + \delta(\tau(x^{(\ell)}_{r_{\ell}}))\equiv 1,
		\end{aligned}
	\right.\text{ that is }
	\left\{
		\begin{aligned}
			\delta(x^{(\ell)}_1) & + \delta(x^{(\ell)}_2)\equiv 1,\\
			&\ \,\vdots\\
			\delta(x^{(\ell)}_{r_{\ell}}) & + \delta(x^{(\ell)}_{1})\equiv 1.
		\end{aligned}
	\right.\pmod 2
	\end{equation}
	Denote by $C_{\ell}$ and $(C_{\ell}\,|\,\mathbf{1})$ the coefficient matrix and the augmented matrix of the $\ell$-th subsystem, where $\mathbf{1}$ is the column vector (of suitable size) with all entries equal $1$, i.e., 
	\[C_{\ell}=\left(\begin{array}{ccccc}
		1 & 1 & & & \\
		& 1 & 1 & & \\
		& & \ddots & \ddots & \\
		& & & 1 & 1 \\
		1 & & & & 1 
	\end{array}\right)_{r_{\ell}\times r_{\ell}}\text{ and }\quad (C_{\ell}|\mathbf{1})=\left(\begin{array}{ccccc|c}
		1 & 1 & & & & 1  \\
		& 1 & 1 & & & 1 \\
		& & \ddots & \ddots & & \vdots \\
		& & & 1 & 1 & 1 \\
		1 & & & & 1 & 1
	\end{array}\right).\]
	Working over $\mathbb{F}_2$, the row echelon forms of $C_{\ell}$ and $(C_{\ell}\,|\,\mathbf{1})$ indicate that for $\ell=1,\dots,h$, 
	\begin{equation}\label{eq:rank-l}
		\mathrm{rank}(C_{\ell})=r_{\ell}-1\quad\text{and}\quad
		\mathrm{rank}(C_{\ell}\,|\,\mathbf{1})=\begin{cases}
			r_{\ell}-1, & \text{ if }r_{\ell}\text{ is even},\\
			r_{\ell}, & \text{ if }r_{\ell}\text{ is odd}.
		\end{cases} 
	\end{equation}
	
	(1) Now we prove the `if' part. Suppose that $\tau$ contains a cycle of odd length, i.e., there exist $j\in\{1,\dots,h\}$ such that $r_j$ is odd. By \eqref{eq:rank-l}, $\mathrm{rank}(C_j)\neq \mathrm{rank}(C_j\,|\,\mathbf{1})$. So the $j$-th subsystem is inconsistent. This implies that the linear system \eqref{odd-2} has no solution.
	
	For the `only if' part, see the proof of the `if' part of Proposition \ref{prop:1}(2).

	(2) Recall that we write the cycle decomposition of $\tau$ as in \eqref{eq:cd-r2} and reorder the linear system \eqref{odd-2} into $h$ subsystems as in \eqref{eq:subsystem-l}. The coefficient matrix of the reordered linear system \eqref{odd-2} with variables arranged in the order \[\delta(x_1^{(1)}),\dots,\delta(x_{r_1}^{(1)}),\delta(x_{1}^{(2)}),\dots,\delta(x_{r_2}^{(2)}),\dots,\delta(x_1^{(h)}),\dots,\delta(x_{r_h}^{(h)})\] is the block  diagonal matrix $C=\mathrm{diag}(C_1,C_2,\dots,C_h)$. Note that 
	\begin{equation}\label{eq:rank-c}
		\mathrm{rank}(C)=\sum_{i=1}^{h}\mathrm{rank}(C_{i})=\sum_{i=1}^{h}(r_{i}-1)=p-1-h.
	\end{equation}
	
	{\it The `if' part.} Assume that $h=k$ and $r_{\ell}$ is even for all $\ell=1,\dots,k$. Then, by \eqref{eq:rank-l}, $\mathrm{rank}(C_{\ell})=\mathrm{rank}(C_{\ell}\,|\,\mathbf{1})=r_{\ell}-1$ for $\ell=1,\dots,k$. So the $\ell$-th subsystem is consistent and it has one free variable. Precisely, for $\ell=1,\dots,k$, the solutions of the $\ell$-th subsystem are 
	\begin{equation}\label{eq:solution}
	\delta(x^{(\ell)}_j) \equiv 1+j+\delta(x^{(\ell)}_1)\pmod{2}, \quad j=1,\dots,r_{\ell}.
	\end{equation}
	In total the linear system \eqref{odd-2} has $k$ free variables and it has exactly $2^k$ solutions. 
	
	{\it The `only if' part.} Assume that the linear system \eqref{odd-2} has $2^k$ solutions. Then all the~$h$ subsystems have to be consistent. So for $\ell=1,\dots,h$, $\mathrm{rank}(C_{\ell})=\mathrm{rank}(C_{\ell}\,|\,\mathbf{1})$ and $r_{\ell}$ is even  by \eqref{eq:rank-l}. Namely, all the cycles in the cycle decomposition of~$\tau$ are of even length. Moreover, when the linear system \eqref{odd-2} has $2^k$ solutions, there are $k$ free variables in $\eqref{odd-2}$ and $\mathrm{rank}(C)=p-1-k$. Then it follows from \eqref{eq:rank-c} that $h=k$. So, in the cycle decomposition of $\tau$, there are exactly $k$ even cycles and no odd cycles.
	\end{proof}

\begin{remark}\label{re:p=9}
Once we have the cycle decomposition of $\tau$, we can read all type II apwenian substitutions from \eqref{eq:solution}. For example, when $p=9$, \[\tau=(1,3,5,7,0,2,4,6,8)=(0,1,3,7,6,4)(2,5).\]
Write $(\delta(0),\delta(2))=(x,y)\in\{0,1\}^{2}$. By \eqref{eq:solution},
\[\delta=(x,\,1-x,\,y,\,x,\,1-x,\,1-y,\,x,\,1-x).\]
Recall that $\delta(j)=\frac{v_j-v_{j+1}}{2}$. While the initial value is $v_0=1$, there are only four type II apwenian substitutions:
\[
\begin{array}{c|*{9}{r}}
(x,y) & v_0 & v_1 & v_2 & v_3 & v_4 & v_5 & v_6 & v_7 & v_8\\[.1em]
\hline
(0,0) & 1& 1 & -1 & -1 & -1 & 1 & -1 & -1 & 1\\
(0,1) & 1 & 1 & -1 & 1 & 1 & -1 & -1 & -1 & 1\\
(1,0) & 1 & -1 & -1 & -1 & 1 & 1 & -1 & 1 & 1\\
(1,1) & 1 & -1 & -1 & 1 & -1 & -1 & -1 & 1 & 1
\end{array}\]
\end{remark}

\subsection{The cycle decomposition of $\tau$ when $p$ is odd}\label{sec:tau}
Let $p\geq3$ be an odd number. Recall that 
\[\mu(p):=\mathrm{ord}_p(2)=\min\{j\in[1,p-1] :  p|(2^{j}-1)\}.\]
By the minimality of $\mu(p)$, we know that $\mu(p)|m$ if $p|(2^m-1)$. Hence,
by Euler's theorem,  $\mu(p)|\phi(p)$, where $\phi(p)$ is the Euler function. In particular, $\mu(p)|(p-1)$ if $p$ is prime.

Recall that $\tau$ is the permutation on $\{0,1,2,\dots,p-2\}$ induced by $j\mapsto 2j+1~(\mathrm{mod}~p):$
\[\tau: \begin{pmatrix}
0& 1 &\cdots& \frac{p-3}{2}& \frac{p-1}{2}& \frac{p+1}{2}& \cdots& p-2\\
1& 3& \cdots& p-2& 0& 2& \cdots& p-3
\end{pmatrix}.\]
We obtain the following explicit formula for the number of cycles in the cycle decomposition of~$\tau$.
\begin{proposition}\label{prop:21}
Let $p\geq 3$ be an odd number. There are $k$ cycles in the cycle decomposition of $\tau$, where
\begin{align}
k & = \frac{1}{\mu(p)}\sum_{j=0}^{\mu(p)-1}\#\{n\in [0,p-2] :  (2^{j}-1)(n+1)\equiv 0\quad (\mathrm{mod}~p)\}\nonumber\\
& = \frac{1}{\mu(p)}\sum_{j=0}^{\mu(p)-1}\gcd(2^{j}-1,p) -1.\label{eq:p-3}
\end{align}
\end{proposition}
\begin{proof}
Since $2^{\mu(p)}\equiv 1~(\mathrm{mod}~p)$, we have for any $m\in\mathbb{N}$, $2^{\mu(p)}m\equiv m~(\mathrm{mod}~p)$. Note that $\tau^j(n)\equiv2^jn+2^j-1~(\mathrm{mod}~p)$ for any $j\geq0$. Hence $\tau^{\mu(p)}=\mathrm{id}$. The subgroup generated by $\tau$ is
\[G:=\{\tau^{0},\tau,\tau^{2},\dots,\tau^{\mu(p)-1}\}.\]
The number of orbits of the mapping
\[\begin{array}{ccc}
G\times\{0,1,2,\dots,p-2\} & \rightarrow & \{0,1,2,\dots,p-2\}\\
g\times x & \mapsto & g(x)
\end{array}\]
	is equal to the number of cycles in the cycle decomposition of $\tau$. Denote by $\mathrm{Fix}(\tau^{j})$ the set of fixed points of $\tau^{j}$.  By P\'{o}lya's enumeration theorem or its special case --- Burnside's lemma \cite[Theorem 3.22]{Rotman},
\begin{align}
k & = \frac{1}{\mu(p)}\sum_{j=0}^{\mu(p)-1}\# \mathrm{Fix}(\tau^{j})\nonumber\\
& =  \frac{1}{\mu(p)}\sum_{j=0}^{\mu(p)-1}\#\{n\in [0,p-2] :  (2^{j}-1)(n+1)\equiv 0\quad (\mathrm{mod}~p)\}.\label{eq:p-2}
\end{align}
For any fixed $j$, let $\gcd(2^{j}-1,p)=:d$. Then
\[2^{j}-1=s\cdot d \quad \text{and}\quad p=t\cdot d\]
where $\gcd(s,t)=1$. For $n\in [0,p-2]$, we have
\[(n+1)\cdot (2^{j}-1)\equiv 0 ~(\mathrm{mod}~ p)\implies (n+1)\cdot(s\cdot d)\equiv 0~(\mathrm{mod}~ p) \implies n+1= m\cdot t.\]
While $1\leq n+1\leq p-1$ and $n+1=mt$, we have $1\leq m\leq d-1$. So
\begin{align*}
& \#\{n\in [0,p-2] :  (n+1)(2^{j}-1)\equiv 0\quad (\mathrm{mod}~p)\}\\
= &  \#\{m\in [1,d-1] :  (m\cdot t)\cdot (s\cdot d)\equiv 0\quad (\mathrm{mod}~p)\}\\
= & d-1 = \gcd(2^j-1,p)-1.
\end{align*}
By \eqref{eq:p-2}, we derive 
\[
k=\frac{1}{\mu(p)}\sum_{j=0}^{\mu(p)-1}(\gcd(2^j-1,p)-1) = \frac{1}{\mu(p)}\sum_{j=0}^{\mu(p)-1}\gcd(2^j-1,p)-1.\qedhere
\]
\end{proof}

It follows from \eqref{eq:p-3} that
\begin{equation}
\label{computaion of k}
k=\frac{1}{\mu(p)}\sum_{d\,|\,p}d\times \#\left\{j\in [0,\mu(p)-1] :  \gcd(2^j-1,p)=d\right\}-1.
\end{equation}
For example, when $p=9$, we have $\mu(9)=6$ and
\[
\begin{array}{r|*{6}{c}}
j & 0 & 1 & 2 & 3 & 4 & 5\\
2^j~(\mathrm{mod}~9)  & 1 & 2 & 4 & 8 & 7 & 5\\
2^j-1~(\mathrm{mod}~9)  & 0 & 1 & 3 & 7 & 6 & 4\\
d  & 9 &  1 &  3 &  1 &  3 &  1
\end{array}
\]
So $k=\frac{1}{6}(1\times 3+3\times 2+9\times 1)-1=2$.

\medskip
Now, we are ready to prove Theorem \ref{thm:D}.
\begin{proof}[Proof of Theorem \ref{thm:D}]
According to Theorem \ref{thm:t2} (1), $N_p$ is equal to the number of solutions of (\ref{odd}) which is equivalent to the linear system (\ref{odd-2}). If $\tau$ contains an odd cycle, then by Proposition \ref{prop:1}(1), $N_p=0$. The remaining case follows from Propositions \ref{prop:1}(2) and \ref{prop:21}.
\end{proof}

\begin{corollary}\label{coro2}
Let $p\geq 3$ be an odd number and the cycle decomposition of $\tau$ has $k$ cycles.
\begin{enumerate}
\item If $p=p_1^{\ell}$, where $p_1$ is prime with $\mu(p_1^2)=\mu(p_1)p_1$ and $\ell\geq1$ is an integer, then $k=\frac{p_1-1}{\mu(p_1)}\ell$.
\item If $p=p_1p_2$, where $p_1,p_2$ are prime and $p_1\neq p_2$, then $k=\frac{(p_1-1)(p_2-1)}{\mathrm{lcm}(\mu(p_1),\mu(p_2))}+\frac{p_1-1}{\mu(p_1)}+\frac{p_2-1}{\mu(p_2)}$ where $\mathrm{lcm}(a,b)$ is the least common multiple of $a$ and $b$.
\end{enumerate}
\end{corollary}
We remark that some primes $p$ do not satisfy $\mu(p^2)=\mu(p)p$, for example, the Wieferich primes~\cite{Wie09}. To prove Corollary \ref{coro2}, we need the following lemma.
\begin{lemma}\label{mu}
Let $p\geq 3$ be prime with $\mu(p^{s})=\mu(p^{s-1})p$ for some integer $s\ge 2$. Let $p_1$ and $p_2$ be prime with $p_1\neq p_2$.
\begin{enumerate}
 \item $\mu(p^\ell)=\mu(p^{\ell-1})p$ for all $\ell\geq s$.
  \item $\mu(p_1p_2)=\mathrm{lcm}(\mu(p_1),\mu(p_2))$.
\end{enumerate}
\end{lemma}
\begin{proof}

(1) The conclusion holds for $\ell=s$. Assume that the conclusion holds for $\ell =k$ where $k\ge s$, we prove the case $\ell=k+1$. 
We first show that $p^{k+1}\nmid(2^{\mu(p^{k})}-1)$. By the inductive assumption, $\mu(p^k)=\mu(p^{k-1})p$. Then $p^{k}\nmid (2^{\mu(p^{k-1})}-1)$ and 
\begin{align*}
	2^{\mu(p^k)}-1 & =2^{\mu(p^{k-1})p}-1 = (2^{\mu(p^{k-1})}-1)P(k),
\end{align*}
where $P(k)=(p+2^{\mu(p^{k-1})}-1+\dots+2^{(p-1)\mu(p^{k-1})}-1)$. When $k=2$, since $p | (2^{\mu(p)}-1)$ and $p^2\nmid (2^{\mu(p)}-1)$, we can write $2^{\mu(p)}=hp+1$ with $p\nmid h$. Then 
\begin{align*}
	P(2) & = \bigl(p+hp+hp(1+2^{\mu(p)})+\dots+hp(1+2^{\mu(p)}+\dots+2^{(p-2)\mu(p)})\bigr)\\
	& = \bigl(p+hp+hp(2+2^{\mu(p)}-1)+\dots+hp(p-1+2^{\mu(p)}-1+\dots+2^{(p-2)\mu(p)}-1)\bigr)\\
	& = \bigl(p+\frac{p-1}{2}p^2h+hp(2^{\mu(p)}-1)+\dots+hp(2^{\mu(p)}-1+\dots+2^{(p-2)\mu(p)}-1)\bigr).
\end{align*}
Note that $p | (2^{m\mu(p)}-1)$ for all $m\geq 1$. We have $p^2\nmid P(2)$. So $p^{3}\nmid(2^{\mu(p^{2})}-1)$. When $k\geq 3$, we have $p^2 | (2^{m\mu(p^{k-1})}-1)$ for all $m\geq 1$. Hence, $p^2\nmid P(k)$. Consequently, $p^{k+1}\nmid(2^{\mu(p^{k})}-1)$.

Since $p^{k}|p^{k+1}$, we have $\mu(p^{k})|\mu(p^{k+1})$. Suppose that $\mu(p^{k+1})=\mu(p^{k})x$ with $x\geq1$. Note that
\begin{align*}
2^{\mu(p^{k})x}-1 
&  = (2^{\mu(p^{k})}-1)(1+2^{\mu(p^{k})}+2^{\mu(p^{k})2}+\cdots+2^{\mu(p^{k})(x-1)})\\
&  = (2^{\mu(p^{k})}-1)(x+2^{\mu(p^{k})}-1+2^{\mu(p^{k})2}-1+\cdots+2^{\mu(p^{k})(x-1)}-1).
\end{align*}
Since $p^{k+1}|(2^{\mu(p^{k})x}-1)$, $p^{k}|(2^{\mu(p^{k})}-1)$ and $p^{k+1}\nmid(2^{\mu(p^{k})}-1)$, then \[p|(x+2^{\mu(p^{k-1})}-1+2^{\mu(p^{k-1})2}-1+\cdots+2^{\mu(p^{k-1})(x-1)}-1).\]
Since $p|(2^{\mu(p^{k-1})m}-1)$ for all $m\geq1$, we have $p|x$. Hence, by the definition of $\mu(p^{k+1})$, we have $x=p$. Finally $\mu(p^{k+1})=\mu(p^{k})p$.

(2) Let $L=\mathrm{lcm}(\mu(p_1),\mu(p_2))$ and suppose $L=\mu(p_1)x=\mu(p_2)y$ where $x,y\in\mathbb{N}$. Then 
\begin{align*}
	2^L-1 & = (2^{\mu(p_1)}-1)(1+2^{\mu(p_1)}+\dots+2^{(x-1)\mu(p_1)})\\
	& = (2^{\mu(p_2)}-1)(1+2^{\mu(p_2)}+\dots+2^{(y-1)\mu(p_2)}).
\end{align*}
Since $p_1|(2^{\mu(p_1)}-1)$, $p_2|(2^{\mu(p_2)}-1)$ and $p_1\neq p_2$, we have $p_1p_2 | (2^L-1)$. Therefore, $\mu(p_1p_2) | L$. 
On the other hand, since $p_1|p_1p_2$ and $p_2|p_1p_2$, we have $\mu(p_1)|\mu(p_1p_2)$ and  $\mu(p_2)|\mu(p_1p_2)$. Hence, $L| \mu(p_1p_2)$. So that $L=\mu(p_1p_2)$.
\end{proof}

\begin{proof}[Proof of Corollary \ref{coro2}]
(1) If $d|p_1^\ell$, where $p_1$ is prime, then $d=p_1^i$ with $0\leq i\leq \ell.$
Note that for every odd number $p'\geq 3$, if $p'|(2^m-1)$, then there exists an integer $k$ such that $m=k\mu(p')$.  Hence, for any fixed  $i\in [0,\ell-1]$, if $p_1^i|(2^j-1)$, then $j=k\mu(p_1^i)$ for some $k$. It follows that \[\#\left\{j\in [0,\mu(p)-1] :  p_1^i|(2^j-1)\right\}=\frac{\mu(p)}{\mu(p_1^i)}.\]

For any fixed $i\in [0,\ell-1]$, by the fact that $\gcd(2^j-1,p)=p_1^i$ if and only if $p_1^i|(2^j-1)$ and $p_1^{i+1}\nmid (2^j-1)$, we have  \[\#\left\{j\in [0,\mu(p)-1] :  \gcd(2^j-1,p)=p_1^i\right\}=\frac{\mu(p)}{\mu(p_1^i)}-\frac{\mu(p)}{\mu(p_1^{i+1})}\] where $\mu(1)=1$.  By Lemma \ref{mu}, for any fixed  $i\in [1,\ell-1]$, we have
$$\#\left\{j\in [0,\mu(p)-1] :  \gcd(2^j-1,p)=p_1^i\right\}=p_1^{\ell-i-1}(p_1-1).$$
Note that \[\#\left\{j\in [0,\mu(p)-1] :  \gcd(2^j-1,p)=p\right\}=1\] and \[\#\left\{j\in [0,\mu(p)-1] :  \gcd(2^j-1,p)=1\right\}=\mu(p)-p_1^{\ell-1}.\] By \eqref{computaion of k},  we have $k=\frac{p_1-1}{\mu(p_1)}\ell$.

(2) If $d|p_1p_2$, where $p_1\neq p_2$ are prime, then $d\in\{1,p_1,p_2,p\}$. Similarly, we have
\begin{align*}
    \#\left\{j\in [0,\mu(p)-1] :  \gcd(2^j-1,p)=p_1\right\}   &  =\frac{\mu(p)}{\mu(p_1)}-1, \\
\#\left\{j\in [0,\mu(p)-1] :  \gcd(2^j-1,p)=p_2\right\}   &  =\frac{\mu(p)}{\mu(p_2)}-1.
\end{align*}
Note that $\#\left\{j\in [0,\mu(p)-1] :  \gcd(2^j-1,p)=p\right\} =1$. We have
$$\#\left\{j\in [0,\mu(p)-1] :  \gcd(2^j-1,p)=1\right\} =\mu(p)-1-\left(\frac{\mu(p)}{\mu(p_1)}-1\right)-\left(\frac{\mu(p)}{\mu(p_2)}-1\right).$$ By \eqref{computaion of k} and Lemma \ref{mu},  we have $k=\frac{(p_1-1)(p_2-1)}{\mathrm{lcm}(\mu(p_1),\mu(p_2))}+\frac{p_1-1}{\mu(p_1)}+\frac{p_2-1}{\mu(p_2)}$.
\end{proof}

The following proposition gives a description of $p$ when there is an odd cycle in the cycle decomposition of $\tau$. 
\begin{proposition}\label{prop:4}
Let $p\geq 3$ be an odd number.  There is an odd cycle in the cycle decomposition of $\tau$ if and only if $p=mp_1$ for some $m\geq 1$, where $\mu(p_1)$ is odd.
\end{proposition}
\begin{proof} 
{\it The `if' part.} If $p=mp_1$ for some $m\geq 1$, then $p|m(2^{\mu(p_1)}-1)$. Hence, $\tau^{\mu(p_1)}(m-1)=m-1$. Since $\mu(p_1)$ is odd, there is an odd cycle in the cycle decomposition of $\tau$.

{\it The `only if' part.} If there is an odd cycle in the cycle decomposition of $\tau$, then there exist $m\geq 1$ and $n\in\{0,1,\dots,p-2\}$ such that $\tau^{m}(n)=n$. So $p|(n+1)(2^{m}-1).$ Note that $n+1\leq p-1$. Hence, there exist $p_1\geq2$ and $p_2\geq1$ such that $p=p_1p_2$ and $p_1|(2^{m}-1)$. Since  $p_1|(2^{\mu(p_1)}-1)$, we have $\mu(p_1)|m$. Hence, $\mu(p_1)$ is odd which completes this proof.
\end{proof}
\begin{remark}
Let $p$ be odd and $N_p$ be the number of apwenian series satisfying \eqref{eq:fp}.  Then, Propositions \ref{prop:21} and \ref{prop:4} imply that 
$$N_p=\begin{cases}
    0,  & \text{if $p=mp_1$ for some $m\geq1$ and $p_1\geq 3$ with $\mu(p_1)$ odd }, \\
    2^k, & \text{otherwise},
\end{cases}$$
where $k=\frac{1}{\mu(p)}\sum_{j=0}^{\mu(p)-1}\gcd(2^{j}-1,p) -1.$
 \end{remark}
 

\section{The \texorpdfstring{$\pm 1$}{+-1}  apwenian sequences generated by substitutions of constant length} \label{sec:gen} 
In this section, we consider all substitutions of constant length. We show that the only substitutions that have a chance to generate $\pm 1$ apwenian sequences are type II substitutions (see Theorem \ref{thm:general}).

A substitution $\sigma$ (of constant length)  on $\{-1,1\}$ is of the form 
\[\sigma:~1\mapsto v_0v_1\cdots v_{p-1},\quad -1\mapsto w_0w_1\cdots w_{p-1}\]
where $p\ge 2$ is an integer, and $v_i$, $w_i\in\{-1,1\}$ for $i=0,1,\dots,p-1$. To ensure that $\sigma$ has a fixed point, we need $v_0=1$ or $w_0=-1$. Let $\mathbf{d}=\sigma^{\infty}(1)$. In addition, we can assume that $d_0=1$ (which also means $v_0=1$). Otherwise, we investigate $-\mathbf{d}$. Define
\[A:=\left\{j\in [0,p-1]  :  v_j=w_j\right\}.\]
 Then $\mathbf{d}$ satisfies the recurrence relations
\begin{equation}
d_0=1 \text{ and } \forall~ n\geq 0,\,
\begin{cases}
d_{np+j}=v_j, & \text{if }j\in A,\\
d_{np+j}=v_jd_n, & \text{otherwise}.
\end{cases} \label{eq:rec-g}
\end{equation}
Note that
\begin{itemize}
\item if $j\not\in A$, then $w_j=\bar{v}_j$;
\item if $\# A = 0$, then $\sigma$ is a type II substitution;
\item if $\# A = p$, then $\mathbf{d}$ is periodic.
\end{itemize}
In fact, our next result shows that substitutions of constant length, which are not type II, have no apwenian fixed point.

\begin{theorem}\label{thm:general}
Let $\mathbf{d}\in\{-1,1\}^{\infty}$ be given by \eqref{eq:rec-g}. If $\mathbf{d}$ is apwenian, then $\# A=0$.
\end{theorem}

Our strategy to prove Theorem \ref{thm:general} is the following:
\begin{enumerate}
\item Show that $0\notin A$ and $(p-1)\notin A$ (Lemma \ref{lem:0-A});
\item Show by induction that every $j\notin A$. In the following diagrams, both $i \rightarrow j$ and $j\leftarrow i$ mean that $i\notin A$ implies $j\notin A.$
When $p$ is odd,  we prove the relation
\[\overbrace{0\to 1\to\cdots \to \left(\frac{p-1}{2}\right)\to\left(\frac{p+1}{2}\right)\to\cdots (p-2) \to (p-1)}^{\text{Lemma \ref{lem:j->j+1}}}.\] 
Then it is reduced to prove $0\notin A$.
When $p$ is even, we prove the relation
\[
\overbrace{0\to 1\to \cdots \to \left(\frac{p}{2}-1\right)}^{\text{Lemma \ref{lem:j->j+1}}} \,\,\nleftrightarrow\,\, \overbrace{\frac{p}{2}\leftarrow \cdots \leftarrow (p-2) \leftarrow (p-1)}^{\text{Lemma \ref{lem:j->j-1}}}.
\]
Then, in this case, we only need to show $0\notin A$ and $p-1\notin A$.
\end{enumerate}
To achieve this, we need some preparation.
\begin{lemma}\label{lem:x4}
Let $\mathbf{d}\in\{-1,1\}^{\infty}$ be given by \eqref{eq:rec-g} and $x\in\{-1,1\}$ is a constant. If $\mathbf{d}$ is apwenian and it is not the Thue-Morse sequence, then none of the following holds: \\
(i) $\forall\, n\geq 0$,  $d_{2n}=xd_{n}$;\\
(ii) $\forall\, n\geq 0$,  $d_{2n+1}=xd_{n}$.
\end{lemma}
\begin{proof}
(i) If for all $n\geq 0$, $d_{2n}=xd_{n}$, then by Theorem \ref{thm:A}, $\mathbf{d}$ is apwenian implies that
\begin{align*}
1 & \equiv \frac{d_{n}+d_{n+1}-d_{2n+1}-d_{2n+2}}{2} \\
	&\equiv \frac{d_{n}-d_{2n+1}}{2}+\frac{1-x}{2}d_{n+1}\\
 & \equiv \frac{d_{n}-d_{2n+1}}{2}+\frac{1-x}{2} ~(\mathrm{mod}~2).
\end{align*}
Hence, $d_{2n+1}=-xd_{n}$ for all $n\geq 0$. Since $d_0=xd_0$, we see $x=1$. Therefore, $\mathbf{d}$ is the Thue-Morse sequence starting with $d_0$.

(ii) If for all $n\geq 0$, $d_{2n+1}=xd_{n}$, then by Theorem \ref{thm:A}, $\mathbf{d}$ is apwenian
yields that $d_{2n+2}=-xd_{n+1}$ for all $n\geq 0$.
If, in addition, $x=-1$, then $\mathbf{d}$ is the Thue-Morse sequence starting with~$1$. If $x=1$, then $d_1=d_0=1$ and for all $k\geq 1$,
\[d_{2^k}=(-1)^{k}d_{0},\quad d_{2^{k}p}=(-1)^{k}d_p,\quad d_{2^{k}p+1}=(-1)^{k-1}d_p.\]
By \eqref{eq:rec-g},
\begin{align*}
(-1)^{k}d_{p} = d_{2^{k}p} = \begin{cases}
v_0, & \text{ if }0\in A,\\
v_0d_{2^{k}}, & \text{ if }0\notin A
\end{cases} & \implies 0\notin A \text{ and }d_{p}=v_0d_0,\\
(-1)^{k-1}d_{p}  = d_{2^{k}p+1} = \begin{cases}
v_1, & \text{ if }1\in A,\\
v_1d_{2^{k}}, & \text{ if }1\notin A
\end{cases} & \implies 1\notin A \text{ and }d_{p}=-v_1d_0.
\end{align*}
So $v_1=-v_0=-1$. Using \eqref{eq:rec-g} again, we obtain $d_1=v_1d_0=-1$, which is a contradiction. \qedhere

\end{proof}

\begin{lemma}\label{lem:0-A}
Let $\mathbf{d}\in\{-1,1\}^{\infty}$ be given by \eqref{eq:rec-g}. If $\mathbf{d}$ is apwenian, then $0\notin A$ and $p-1\notin A$.
\end{lemma}
\begin{proof}
Note that for the Thue-Morse sequence $A=\emptyset$. Therefore we may assume that $\mathbf{d}$ is not the Thue-Morse sequence. If $0\in A$, then by Theorem \ref{thm:A}, $\mathbf{d}$ is apwenian means for all $n\geq 0$,
\begin{align*}
1 & \equiv \frac{d_{np+p-1}+d_{(n+1)p}-d_{(2n+1)p+p-1}-d_{(2n+2)p}}{2}\quad (\mathrm{mod}~2)\\
& \equiv \begin{cases}
\frac{v_{p-1}+v_0-v_{p-1}-v_0}{2}, & \text{ if } p-1\in A\\
\frac{v_{p-1}d_{n}+v_0-v_{p-1}d_{2n+1}-v_0}{2}, & \text{ if }p-1\notin A,
\end{cases}\quad (\mathrm{mod}~2)\\
& \equiv
\begin{cases}
 0, & \text{ if } p-1\in A\\
\frac{d_{n}-d_{2n+1}}{2}, & \text{ if }p-1\notin A,
\end{cases}\quad (\mathrm{mod}~2).
\end{align*}
When $p-1\in A$, we have a contradiction $1\equiv 0$. In the latter case, we have $d_{2n+1}=-d_{n}$ for all $n\geq 0$. By Lemma \ref{lem:x4}, $\mathbf{d}$ is not apwenian which is also a contradiction. So $0\notin A$.

The same kind of argument as above shows that $p-1\notin A$.
\end{proof}

\begin{lemma}\label{lem:j->j+1}
Let $\mathbf{d}\in\{-1,1\}^{\infty}$ be given by \eqref{eq:rec-g}. If $\mathbf{d}$ is apwenian, then   $j\notin A$ and $j\neq\frac{p}{2}-1$ imply $j+1\notin A$.
\end{lemma}
\begin{proof}
Suppose $j\notin A$. If $j+1\in A$, then by Theorem \ref{thm:A}, $\mathbf{d}$ is apwenian yields that for all $n\geq 0$,
\begin{equation}\label{eq:25-1}
1\equiv \frac{v_jd_n+v_{j+1}-d_{2np+2j+1}-d_{2np+2j+2}}{2} \quad (\mathrm{mod}~2).
\end{equation}
Note that $j\neq\frac{p}{2}-1$ implies either $2j+1<p-1$ or $2j+1\geq p$.  Hence, there are two cases.

(1) If $2j+1<p-1$, then $2j+2\leq p-1.$ There are four sub-cases.

{\it Case 1}: $2j+1\in A$ and $2j+2\in A$. In this case, $d_{2np+2j+1}=v_{2j+1}$ and  $d_{2np+2j+2}=v_{2j+2}$. By \eqref{eq:25-1}, for all $n\geq 0$,
\[1\equiv \frac{v_jd_n+v_{j+1}-v_{2j+1}-v_{2j+2}}{2} \quad (\mathrm{mod}~2).\]
So $\mathbf{d}$ is a constant sequence which is not apwenian.

{\it Case 2}: $2j+1\in A$ and $2j+2\notin A$.  In this case, $d_{2np+2j+1}=v_{2j+1}$ and  $d_{2np+2j+2}=v_{2j+2}d_{2n}$. By \eqref{eq:25-1}, for all $n\geq 0$,
\[1\equiv \frac{v_jd_n+v_{j+1}-v_{2j+1}-v_{2j+2}d_{2n}}{2} \quad (\mathrm{mod}~2).\]
We have $d_{2n}=xd_n$ for all $n\geq 0$ where $x\in\{-1,1\}$. By Lemma \ref{lem:x4}, $\mathbf{d}$ is not apwenian. 

{\it Case 3}: $2j+1\notin A$ and $2j+2\in A$. The same proof as in Case 2 works.

{\it Case 4}: $2j+1\notin A$ and $2j+2\notin A$. In this case, $d_{2np+2j+1}=v_{2j+1}d_{2n}$ and  $d_{2np+2j+2}=v_{2j+2}d_{2n}$.  By \eqref{eq:25-1}, for all $n\geq 0$,
\begin{align*}
1 & \equiv \frac{v_jd_n+v_{j+1}-v_{2j+1}d_{2n}-v_{2j+2}d_{2n}}{2} \\
& \equiv \frac{v_jd_n+v_{j+1}-v_{2j+1}-v_{2j+2}}{2} \quad (\mathrm{mod}~2).
\end{align*}
We see that $\mathbf{d}$ is a constant sequence which is not apwenian. 

(2) If $2j+1\geq p$, we will consider whether $2j+1-p,2j+2-p$ belong to $A$ or not. In this case, we have similar four cases. We omit the proof here.
\end{proof}

\begin{lemma}\label{lem:j->j-1}
Let $\mathbf{d}\in\{-1,1\}^{\infty}$ be given by \eqref{eq:rec-g}. If $\mathbf{d}$ is apwenian, then $j\notin A$ and $j\neq \frac{p}{2}$ imply $j-1\notin A$.
\end{lemma}
\begin{proof}
Suppose $j\notin A$. If $j-1\in A$, then by Theorem \ref{thm:A}, $\mathbf{d}$ is apwenian yields that for all $n\geq 0$,
\begin{equation}\label{eq:26-1}
1\equiv \frac{v_{j-1}+v_{j}d_n-d_{2np+2j-1}-d_{2np+2j}}{2} \quad (\mathrm{mod}~2).
\end{equation}

Note that $j\neq\frac{p}{2}$ implies either $2j-1<p-1$ or $2j-1\geq p$.  Hence, there are two cases.

(1) If $2j-1<p-1$, then $2j\leq p-1.$ There are four sub-cases.

{\it Case 1}: $2j-1\in A$ and $2j\in A$. By \eqref{eq:25-1}, for all $n\geq 0$,
\[1\equiv \frac{v_{j-1}+v_{j}d_n-v_{2j-1}-v_{2j}}{2} \quad (\mathrm{mod}~2).\]
So $\mathbf{d}$ is a constant sequence which is not apwenian.

{\it Case 2}: $2j-1\in A$ and $2j\notin A$. By \eqref{eq:25-1}, for all $n\geq 0$,
\[1\equiv \frac{v_{j-1}+v_{j}d_n-v_{2j-1}-v_{2j}d_{2n}}{2} \quad (\mathrm{mod}~2).\]
We have $d_{2n}=xd_n$ for all $n\geq 0$ where $x\in\{-1,1\}$. By Lemma \ref{lem:x4}, $\mathbf{d}$ is not apwenian.

{\it Case 3}: $2j-1\notin A$ and $2j\in A$. The same proof as in Case 2 works.

{\it Case 4}: $2j-1\notin A$ and $2j\notin A$. By \eqref{eq:25-1}, for all $n\geq 0$,
\begin{align*}
1 & \equiv \frac{v_{j-1}+v_{j}d_n-v_{2j-1}d_{2n}-v_{2j}d_{2n}}{2} \\
& \equiv \frac{v_{j-1}+v_{j}d_n-v_{2j-1}-v_{2j}}{2} \quad (\mathrm{mod}~2).
\end{align*}
We see that $\mathbf{d}$ is a constant sequence which is not apwenian. 

(2) If $2j+1\geq p$, we have similar discussion.  We omit the details here.
\end{proof}

Now we are ready to prove Theorems \ref{thm:general} and \ref{thm:B}.
\begin{proof}[Proof of Theorem \ref{thm:general}]
Suppose that $\mathbf{d}$ is apwenian. By Lemma \ref{lem:0-A}, we have $0\notin A$ and $p-1\notin A$. Then by Lemma \ref{lem:j->j+1}, $0\notin A$ gives that $j\notin A$ for all $j<\frac{p}{2}$. On the other hand, by Lemma \ref{lem:j->j-1}, it follows from $p-1\notin A$ that $j\notin A$ for all $j>\frac{p}{2}-1$.
\end{proof}

\begin{proof}[Proof of Theorem \ref{thm:B}]
	According to Theorem \ref{thm:general}, if $\mathbf{d}$ is apwenian, then 
	$\mathbf{d}$ satisfies the recurrence relations for all $n\geq 0$ and for all $j=0,1,\dots,p-1$,
	\[d_{np+j}=v_jd_n.\]
	Therefore, 
	\begin{align*}
		f(z) & = \sum_{n=0}^{\infty}\bigl(d_{np}+d_{np+1}z+\cdots+d_{np+(p-1)}z^{p-1}\bigr)z^{np}\\
		& = \sum_{n=0}^{\infty}\bigl(v_0d_{n}+v_1d_{n}z+\cdots+v_{p-1}d_{n}z^{p-1}\bigr)z^{np}\\
		& = \bigl(v_0+v_1z+\cdots+v_{p-1}z^{p-1}\bigr)f(z^p),
	\end{align*}
	which implies that $f(z)$ satisfies \eqref{eq:fh}.
\end{proof}


\section{Discussion and more examples}\label{sec:dis}
\subsection{Substitutions of non-constant length}
The question of finding apwenian sequences in fixed points of substitutions of constants length have been settled in the previous sections. It is natural to ask  the same question for fixed points of substitutions of non-constant length. A countable class of fixed points of substitutions of non-constant length are Sturmian.  Actually, there are uncountably many Sturmian sequences. The next result shows that they are not apwenian.

\begin{proposition}\label{th:Sturmian}
All Sturmian sequences on $\{-1,1\}$ or $\{0,1\}$ are not apwenian.
\end{proposition}
\begin{proof}
Let $\mathbf{d}$ be a Sturmian sequence on $\{a,b\}$ and $\mathcal{F}$ be the set of all subwords of $\mathbf{d}$.

{\it Case 1}: $\{a,b\}=\{1,-1\}$.
Let $\mathbf{d}$ be of type $a$. That is $aa\in\mathcal{F}$ and $bb\notin\mathcal{F}$. If $\mathbf{d}$ is apwenian, then by Theorem \ref{thm:A}, we have $\frac{d_n+d_{n+1}}{2}\equiv 1+\frac{d_{2n+1}+d_{2n+2}}{2}~(\mathrm{mod}~2)$ for all $n\geq0$. So 
\begin{align}
& d_nd_{n+1}=aa \text{ for some } n  \implies d_{2n+1}d_{2n+2}\in\{ab,ba\},\label{eq:aaaba}\\
& d_nd_{n+1}\in\{ab,ba\} \text{ for some } n  \implies d_{2n+1}d_{2n+2}=aa.\label{eq:aaabb}
\end{align}
Since $bb\notin\mathcal{F}$, we see $aba\in\mathcal{F}$. Therefore, $d_{j}d_{j+1}d_{j+2}=aba$ for some $j\ge 0$. By \eqref{eq:aaabb}, we have $d_{2j+1}d_{2j+2}d_{2j+3}d_{2j+4}=aaaa$. By \eqref{eq:aaaba}, we have $d_{4j+3}d_{4j+4}\in\{ab,ba\}$ and $d_{4j+5}d_{4j+6}\in\{ab,ba\}$. Hence, 
\[\big||d_{2j+1}d_{2j+2}d_{2j+3}d_{2j+4}|_a-|d_{4j+3}d_{4j+4}d_{4j+5}d_{4j+6}|_a\big|=2\]
which contradicts the fact that $\mathbf{d}$ is balanced. So $\mathbf{d}$ cannot be apwenian.

{\it Case 2}: $\{a,b\}=\{0,1\}$. If $\mathbf{d}$ is apwenian, then $d_0=1$ and by Theorem \ref{thm:E}, for all $n\geq 0$,
\begin{equation}\label{eq:stur-01}
d_n\equiv d_{2n+1}+d_{2n+2}\quad (\mathrm{mod}~2).
\end{equation}
When $\mathbf{d}$ is of type $0$, by \eqref{eq:stur-01}, $d_n=0$ implies that $d_{2n+1}d_{2n+2}=00$. Consequently, $0^m\in\mathcal{F}$ for any $m\geq1$. Since $\mathbf{d}$ is not (eventually) periodic and $11\notin\mathcal{F}$, there exists an $n_0\ge 1$ such that $10^{n_0}1\in\mathcal{F}$. Therefore, $0^{n_0+2},\,10^{n_0}1\in\mathcal{F}$ which contradicts the fact that $\mathbf{d}$ is balanced. 

When $\mathbf{d}$ is of type $1$, applying \eqref{eq:stur-01} for $n=0$ and $n=1$, we obtain that $d_0d_1d_2d_3d_4\in\{11010,10111\}$ which implies $010\in\mathcal{F}$ or $111\in\mathcal{F}$. If $010\in\mathcal{F}$, then there exists $n\geq0$ such that $d_nd_{n+1}d_{n+2}=010$. By \eqref{eq:stur-01}, this implies that $(d_{2n+i})_{i=1}^{6}\in\{111011,110111\}$. That is to say if $010\in\mathcal{F}$ then $111\in\mathcal{F}$ which contradicts the fact that $\mathbf{d}$ is balanced. Hence, $010\notin\mathcal{F}$. Similarly, $111\notin\mathcal{F}$. Therefore, $\mathbf{d}$ cannot be apwenian.
\end{proof}
The numerical experiment suggests that the fixed points of substitutions of non-constant length (on $\{0,1\}$ or $\{-1,1\}$) can not be apwenian. We propose the following conjecture.
\begin{conjecture}
	The fixed points of substitutions of non-constant length on $\{0,1\}$ or $\{-1,1\}$ can not be apwenian.
\end{conjecture}

\subsection{Hankel determinant of conjugate}
Let $\mathbf{d}\in\{-1,1\}^{\infty}$. We have seen that $\mathbf{d}$ is apwenian if and only if $-\mathbf{d}$ is apwenian. For the sequence $\mathbf{c}\in\{0,1\}^{\infty}$, apparently $-\mathbf{c}\not\in\{0,1\}^{\infty}$. Since $-\mathbf{d}=\bar{\mathbf{d}}$, we also would like to investigate the Hankel determinant of $\bar{\mathbf{c}}$. The following proposition gives a relation between $H_{n}(\mathbf{c})$ and $H_n(\bar{\mathbf{c}}),$ where $\bar{\mathbf{c}}$ is the conjugate of $\mathbf{c}$ defined by $\bar{\mathbf{c}}=(1-c_i)_{i\geq 0}$. 
\begin{proposition}
	Let $\mathbf{c}\in \{0,1\}^\infty$.
	Then,
$H_{n}(\mathbf{c})+H_n(\bar{\mathbf{c}})\equiv 1 ~ (\mathrm{mod}~2)$ for all $n\geq0$ if and only if the sequence $(c_i+c_{i+2})_{i\geq0}\pmod{2}$ is apwenian.
\end{proposition}
\begin{proof}
Write $\mathbf{v}=v_0v_1v_2\cdots\in\{0,1\}^{\infty}$ and $v_i\equiv c_i+c_{i+2}~ (\mathrm{mod}~2)$ for all $i\geq 0$. The Hankel determinant of $\mathbf{c}$ and its conjugate are related in the following way:
for all $n\geq 2$, \[H_{n}(\mathbf{c})+H_n(\bar{\mathbf{c}})\equiv H_{n-1}(\mathbf{v}) \quad (\mathrm{mod}~2).\]
In fact, applying certain column operations and row operations, we have
\begin{align*}
H_n(\bar{\mathbf{c}}) & = \left|\begin{matrix}
1-c_0 & 1-c_1 & \cdots & 1-c_{n-1}\\
1-c_1 & 1-c_2 & \cdots & 1-c_{n}\\
\vdots & \vdots & \ddots & \vdots\\
1-c_{n-1} & 1-c_n & \cdots & 1-c_{2n-2}\\
\end{matrix}\right| = \left|\begin{matrix}
c_1-c_0 &  \cdots & c_{n-1}-c_{n-2} & 1-c_{n-1}\\
c_2-c_1 &  \cdots & c_n-c_{n-1} & 1-c_{n}\\
\vdots  & \cdots & \ddots  &  \vdots\\
c_n-c_{n-1}  & \cdots & c_{2n-2}-c_{2n-3} &  1-c_{2n-2}\\
\end{matrix}\right| \\
& = \left|\begin{matrix}
c_1-c_0 &  \cdots & c_{n-1}-c_{n-2} & -c_{n-1}\\
c_2-c_1 &  \cdots & c_n-c_{n-1} & -c_{n}\\
\vdots  & \cdots & \ddots  &  \vdots\\
c_n-c_{n-1}  & \cdots & c_{2n-2}-c_{2n-3} &  -c_{2n-2}\\
\end{matrix}\right| +
\left|\begin{matrix}
c_1-c_0 &  \cdots & c_{n-1}-c_{n-2} & 1\\
c_2-c_1 &  \cdots & c_n-c_{n-1} & 1\\
\vdots  & \cdots & \ddots  &  \vdots\\
c_n-c_{n-1}  & \cdots & c_{2n-2}-c_{2n-3} &  1\\
\end{matrix}\right| \\
& = \left|\begin{matrix}
-c_0 &  \cdots & -c_{n-2} & -c_{n-1}\\
-c_1 &  \cdots & -c_{n-1} & -c_{n}\\
\vdots  & \cdots & \ddots  &  \vdots\\
-c_{n-1}  & \cdots & -c_{2n-3} &  -c_{2n-2}\\
\end{matrix}\right| +
\left|\begin{matrix}
c_1-c_0 &  \cdots & c_{n-1}-c_{n-2} & 1\\
c_2+c_0 &  \cdots & c_n+c_{n-2} & 0\\
\vdots  & \cdots & \ddots  &  \vdots\\
c_n+c_{n-2}  & \cdots & c_{2n-2}+c_{2n-4} &  0\\
\end{matrix}\right| \\
& = (-1)^n H_{n}(\mathbf{c})+(-1)^{n+1}H_{n-1}(\mathbf{v}). \qedhere
\end{align*}
\end{proof}

\subsection{Examples}
We end this section by some examples, which contain sequences given by substitutions of non-constant length or sequences given by substitutions with projection.
From Theorem \ref{thm:E}, one can easily obtain $0$-$1$ apwenian sequences which are not of type I.
\begin{example}
Assume that $\mathbf{c}$ satisfies \eqref{eq:01}. If we take
	$$(c_{2n+1})_{n\geq 0}=(1,0,1,1,0,1,0,1,1,0,1,1,0,\dots)$$
	to be the Fibonacci sequence given by $1\mapsto 10$, $0\mapsto 1$, then by Theorem \ref{thm:E}, we see that
	$$\mathbf{c}=(1,1,0,0,1,1,1,1,1,0,1,1,0,0,1,1,0,\dots)$$
	is an apwenian sequence. Further, by Theorem \ref{thm:E}, any $0$-$1$ sequence $(c_{2n+1})_{n\geq 0}$ can be uniquely extended to an apwenian sequence $\mathbf{c}$.
\end{example}
Certainly, there are apwenian sequences generated by substitutions with projections.

\begin{example}
	Consider the Thue-Morse substitution \[1\mapsto 1\bar 1,\quad \bar 1\mapsto \bar 11,\] where $\bar 1:=-1.$ Its fixed point starting with $1$ is the Thue-Morse sequence $\mathbf{t}$. Now, we apply the morphism \[1\mapsto 11,\quad \bar 1\mapsto \bar 1\bar 1\] to $\mathbf{t}$,
and obtain the sequence
	$$\mathbf{d}=(1,1,\bar{1},\bar{1},\bar{1},\bar{1},1,1,\bar{1},\bar{1},1,1,1,1,\bar{1},\bar{1},\dots).$$
	Then $\mathbf{d}$ is a $\pm 1$ apwenian sequence and it can not be generated by a type II substitution. \end{example}
\begin{proof}
By the definition of $\mathbf{d}$, we know that for all $n\geq 0$,
\begin{equation}
d_{4n}=d_{4n+1}=d_{2n}\quad \text{ and }\quad d_{4n+2}=d_{4n+3}=-d_{2n}.\label{eq:nt2}
\end{equation}
Since $d_{n}\in\{-1,1\}$, using \eqref{eq:nt2}, one can check that $\mathbf{d}$ satisfies \eqref{eq:-01}. For instance, when $n=4k$, 
\[\frac{d_{4k}+d_{4k+1}-d_{8k+1}-d_{8k+2}}{2}=\frac{d_{4k}+d_{4k+1}-d_{4k}+d_{4k}}{2}=\frac{d_{2k}+d_{2k}}{2}=d_{2k}\equiv 1\pmod 2.\]
Hence, by Theorem \ref{thm:A}, $\mathbf{d}$ is $\pm 1$ apwenian. Then it follows from Theorem \ref{thm:t2}(2) that $\mathbf{d}$ can not be the fixed point of a type II substitution of even length. In the following, we show that $\mathbf{d}$ also can not be the fixed point of a type II substitution of odd length.

Suppose on the contrary that $\mathbf{d}$ is the fixed point of a type II substitution of odd length. Then there exist $p\geq 2$, $v_0,v_1,\dots,v_{p-1}\in\{-1,1\}$ such that for all $n\geq 0$ and $j=0,1,\dots, p-1$,
\begin{equation}\label{eq:nt1}
d_{np+j}=v_jd_n.
\end{equation}
Since $d_0=d_1=1$, by using \eqref{eq:nt1}, a direct computation shows that $v_0=v_1=1$. 
Set $p=2q+1$. Then by \eqref{eq:nt1},
\[d_{p^2}=v_0d_p=v_0^2d_1=1 \text{ and }d_{p^2+1}=v_1d_{p}=v_1v_0d_{1}=1.\]
By \eqref{eq:nt2}, we have
\[1=d_{p^2+1}=d_{4(q^2+q)+2}=-d_{2(q^2+q)}=-d_{4(q^2+q)+1}=-d_{p^2}=-1,\]
which is a contradiction.
\end{proof}

\begin{example}
Consider the substitution \[1\mapsto \bar 11,\, \bar1\mapsto 1\bar1.\] Although it has no fixed point, it yields an apwenian sequence in the following way. Let $\iota$ be a substitution on $\{a,1,-1\}$ given by
\[a\mapsto a1,\quad 1\mapsto \bar11,\quad \bar1\mapsto 1\bar1\]
and $\rho$ be a coding
\[a\mapsto 1,\quad 1\mapsto 1,\quad \bar1\mapsto \bar1.\]
Then
	$$\mathbf{d}=\rho(\iota^{\infty}(a)) = (1,1,\bar{1},1,1,\bar{1},\bar{1},1,\bar{1},1,1,\bar{1},1,\bar{1},\bar1,1,\dots)$$
	is apwenian. In fact, $d_0=1$ and for all $n\geq 0$, $d_{2n+1}=d_n,\,d_{2n+2}=-d_{n+1}$ which fulfill the criterion in Theorem \ref{thm:A}.
\end{example}

\section{Concluding remarks}
After we uploaded the first version of this paper to arXiv, J.-P. Allouche kindly informed us that the permutation $\tau$ (defined in subsection \ref{sec:tau}) has already been considered up to a small change of notation (see \cite[Proposition 3]{A83} and \cite[Eq.~5]{EKF00}). In \cite{A83} and \cite{EKF00}, the authors studied the permutation \[\tau': \begin{pmatrix}
	1 & 2 & \cdots & n & n+1 & n+2 & \cdots & 2n\\
	2 & 4 & \cdots & 2n & 1 & 3 & \cdots & 2n-1
\end{pmatrix}\]
and they showed that the number of cycles of $\tau'$ is $\sum\limits_{d|(2n+1),\, d\neq 1}\frac{\phi(d)}{\textrm{ord}_{d}(2)}$ which is equal to $k$ in Theorem \ref{thm:D}.

Quite recently, Allouche, Han and Niederreiter \cite{AHN20} found a connection between  $0$-$1$ apwenian sequences and sequences with perfect linear complexity profile (PLCP) which were defined more that thirty years ago in the study of measures of randomness for binary sequences. For details of PLCP sequences, see for example \cite{AHN20,Ni88,WM86} and references therein. 
\pagebreak[2]

\begin{center}
	\textsc{Acknowledgement}
\end{center}

We warmly thank the referees for very precise reports on the previous version of this paper. We appreciate J.-P. Allouche, D. Badziahin and J. Shallit for their valuables comments and suggestions.    

This work is supported by the Fundamental Research Funds for the Central Universities from HZAU (2662019PY022) and from SCUT (2020ZYGXZR041). W. Wu is supported by Guangdong Basic and Applied Basic Research Foundation (2021A1515010056) and Guangzhou Science and Technology program (202102020294).

\end{document}